\documentclass[12pt]{article}
\pdfoutput=1
\usepackage{amsfonts,pdfsync,amsmath,hyperref}
\usepackage{amssymb}
\usepackage{amsthm} 
\usepackage[english]{babel}
\usepackage{color,tikz,multicol}
\usepackage[applemac]{inputenc}

\def\dfrac#1#2{{\displaystyle {#1 \over #2}}}
\newcommand{\R}{\mathbb{R}}

\theoremstyle{definition}
\newtheorem{theorem}{Theorem}
\newtheorem{lemma}{Lemma}

\theoremstyle{definition}
\newtheorem{definition}{Definition}

\theoremstyle{definition}
\newtheorem{rem}{Remark}

\usepackage{epstopdf}
\graphicspath{{data/}}

\def\ss{ \scriptsize}

\def\etal{\textit{et.\,al.\.}}

\def\eg{\textit{e.\.g.\.}}

\def\tauN{\tau_N}
\def\tauC{\tau_C}
\def\Nin{ N_{\rm in} }
\def\Cin{ C_{\rm in} }
\def\sigmaN{ \sigma_N }
\def\sigmaC{ \sigma_C }
\def\sigmaM{ \sigma_M }
\def\tauD{ \tau_D }
\def\tauM{ \tau_M }
\def\QCMin{ Q^C_{\rm min} }
\def\QNMin{ Q^N_{\rm min} }

\usepackage{esdiff}
\renewcommand{\d}[1]{\, \mathrm{d} #1}
\def \dt {\d{t}}
\def \ds {\d{s}}

\def\A{\mathbb{A}}

\usepackage{graphicx,textcomp}
\usepackage{graphics,cancel,wasysym}

\begin{document}


\title{A mathematical model of marine mucilage, the case of the \textit{liga} on  the Basque coast}

\author{Charles PIERRE\thanks{charles.pierre@univ-pau.fr,\ LMAP UMR- CNRS 5142, IPRA BP 1155, 64013 Pau Cedex, France.} \and Guy VALLET\thanks{guy.vallet@univ-pau.fr,\ LMAP UMR- CNRS 5142, IPRA BP 1155, 64013 Pau Cedex, France.}}

\date{January 9, 2023}
\maketitle
\abstract{In this paper we are interested in modeling the production of mucus by diatoms under the constraint of a nutrient limitation. The initial questioning comes from the observation of the so-called "liga" on the Aquitaine coast. 
The biological origin of the phenomenon is presented and discussed based on the existing litterature.
A mathematical model is proposed and its theoretical properties are analized: well-posedness and differentiability with respect to the model parameters.
Finally, numerical experiments are provided, investigating the possibility of parameter identification for the model using chemostat-type experiments.
\\
 AMScodes : 37N25,
92D25,
92D40.
}

\section{Introduction}
The presence of pelagic muds composed of  marine mucilage is well documented in the Adriatic sea (see \textit{e.g.} \cite{Alcoverro,CALVO199523,Degobbis,Mecozzia,DELNEGRO2005258}) and the sea of Marmara \cite{Aktan,Balkis}, it is also observed on the Basque coast of  the Atlantic ocean, where it is locally called \textit{liga} \cite{ifremer}. 
Formerly occasionally observed on the Basque coast  in spring these last few decades,  the frequency, abundance and duration of liga have been steadily increasing since the early 2000's.  
This sharp increase of the phenomenon is reported by fishermen, whose nets get clogged by liga. 
This degradation of water quality shows a coastal ecosystem in poor health. 
Excessive marine mucilage production is related to blooms of phytoplankton algae, which  can be harmful (in some cases leading to toxin releases) and can have negative impacts on the local ecosystem, as well as on fishing and aquaculture activities.
\\[0.4cm] \indent
As indicated in \cite{Susperregui},  liga is a term of Basque origin derived from \textit{"ligarda"}: a sticky liquid. 
Liga is a pelagic marine mucilage of thick and gelatinous aspect forming macroscopic aggregates. 
Aggregation is caused by \textit{TEP} (Transparent Exopolymer Particles): extra-cellular organic particles, see \cite{Mari}.
The presence of such aggregates is highly correlated with a class of algae: diatoms, a unicellular phytoplankton species. 
Actually, large concentration of various species of diatoms are systematically present in liga samples as analyzed in \cite{Bussard,Susperregui},
the relationship between liga, TEP and diatom blooms will be discussed further  in section \ref{Diatom-blooms}.
\\
Liga formation occurs in areas where salt and fresh waters mix together. 
Early liga production stage corresponds to pelagic colloidal aggregates, with the appearance of a mud as observed on fishing nets where it accumulates.
Under stable anticyclonic conditions, it has time to mature to a pelagic mucous mass that may migrate in the water column. 
Depending on hydro-climatic conditions and on its biological composition, it  may either settle the sea bottom as a mucous mass or reach the surface as floating mosses. 
Liga forms temporary pelagic ecosystems containing a large diversity of microorganisms (viruses, bacteria, phyto and zoo-plankton). 
Two blooms of liga are usually observed  every year: in spring and autumn. 
It disappears with the first winter storms but can reappear with a low intensity during  winter anticyclonic conditions.
\\[0.4cm] \indent
Liga production scenario is the following.
Under suitable environmental conditions, a bloom (or an important increase) of  diatoms occurs, followed by mucilage secretion. 
When the density of population and the mucus concentration are important enough, it can aggregate and form liga, see \textit{e.g.} \cite[Supplementary Material S2)]{Mari}).
\\
Research on mucilage production by diatoms such as \cite{Reynolds:2007gt} indicates that it might be induced by an imbalance in intracellular nutrient concentrations  and be no more than a homeostatic mechanism to balance cell stoichiometry.
This assumption is discussed in section \ref{sec:stoichiometry} and an imbalance in nutrients C:N ratio is put forward.
\\
The aim of this paper is to propose a model of liga production by diatom populations. 
The assumption of a C:N imbalance as a trigger of mucilage production is considered and a two nutrient model will be proposed based on carbon and nitrogen uptakes by diatoms. 
\\[0.4cm] \indent
This paper is organized as follows.
Section \ref{diatoms} is devoted to a presentation of diatoms, their blooms and the causes of liga.
A  mathematical model of liga production is presented in section \ref{sec:math-mod}.
This will be followed by a mathematical study of this model: existence and uniqueness of the solution, some qualitative information and the stability of the solution with respect to the parameters of the model. 
The last section will be devoted to the numerical study of
the model. Firstly a biological setting of the model inducing mucilage production is introduced. Then the ability to perform a parameter identification of the model based on chemostat experiments is analyzed.


\section{Diatoms}
\label{diatoms}
Diatoms are brown unicellular phytoplankton algae  with a siliceous shell called frustule, a bipartite silica cell wall that encloses the eukaryotic protoplast. It protects the protoplast and provides routes for nutrient uptake, gaseous exchanges and  extra-cellular secretions.
\\
In most cases, diatom life cycle alternates  a vegetative phase (lasting months or years) and a rejuvenation phase (few days) usually preceded by sexual reproduction  \cite{Scala,Seckbach}.
\\
Diatoms have one or two cytoplasmic vacuoles with multiple roles. 
They have reserves that are easily mobilized storage in view of less favorable periods or shortages.
The frustule  provides a supporting structure to the vacuole and protects nutrients from competitors for subsequent cell divisions as noticed in  \cite{SMETACEK199925}.

\subsection{Liga, TEP and diatom blooms }
\label{Diatom-blooms}
As mentioned in the introduction, mucilage aggregation is caused by TEP\footnote{TEP: Transparent Exopolymeric Particles}, see \cite{Mari}.
Diatom vacuoles constitute an important supply of nutrients, in particular $\beta-1,3$-glucan  chrysolaminarin, a carbohydrate, which, according to \cite{Sayanova}, contributes to the production of extracellular polymers.
Indeed, diatoms exude \textit{EPS} (Extra-cellular Polymeric Substances) constituting, under suitable conditions, precursors for the production of TEP, which (mixed with other organic matter particles) then transform to marine snow \cite{Genzer,Mari,Passow}.
TEP are buoyant particles that should stay suspended in surface waters, and sinking is related to the size, the porosity and the density of aggregates. 
Ballasting can be due to  incorporation of heavy minerals or atmospheric dusts \cite{Mari}, or modern pollution such as residue from spilled oil  \cite{Burd,Genzer}.
In \cite{Bartual}, the authors noticed that diatoms also release poly-unsaturated aldehydes during blooms which contribute to increase TEP sizes and generate larger aggregates. 
\\[0.2cm] \indent
On the Basque coast, a first episode of liga usually appears in spring, a second in autumn, with mainly a bloom of diatoms \cite{Susperregui}. It has been reported in \cite{Susperregui2} that the main species of diatoms found in bloom  situations of liga are \textit{Ceratoneis closterium} (always observed), \textit{Leptocylindus danicus} (in June 2013), \textit{Pseudo Nitzschia group B 2 \& B1} (in October 2013) and \textit{Thalassiosira gravida} (in march 2014). In case of a strong presence of liga, the following species have been found too: diatoms (in addition to those indicated above \textit{Navicula spp.},  \textit{Guinardia flaccida}, \textit{Leptocylindrus minimus}, \textit{Guinardia delicatula}, \textit{Rhizosolenia setigera}); dinoflagelates (\textit{Protoperidinium oceanicum});  
gelatinous zooplankton (\textit{Siphonophora Diphiidae}, \textit{Oikopleura sp.}, \textit{Sagitta sp}).
\\[0.4cm] \indent
Diatom concentration is generally higher in the beginning of spring and in autumn, when nutrients start to be abundant again and when light intensity and day length are favorable. 
Furthermore, algal blooms are known to be often initiated by a diatom bloom, as observed  in the Adriatic, where diatom \textit{Cylindrotheca Closterium} is involved in extracellular carbohydrates  production  under nutrient limitation, see \cite{Alcoverro}. 
Several reasons can be pointed out for that.
Firstly, \cite{Allen:2011zg} highlighted a functional cycle of urea conferring to diatoms an ecological advantage during periods of nitrogen limitations to other phytoplankton taxa \cite{Bussard,Levitan,Smith}.
Moreover, diatom brown color reveals the presence of carotenoid, utilized together with chlorophyll \textit{a} and \textit{c}, this ensures a large  photosynthetic light harvesting. 
Eventually, diatoms are capable of C4 photosynthesis,  allowing a more efficient utilization of available CO$_2$, in particular in less favorable conditions  \cite{Scala}. 
\\[0.4cm] \indent
Diatom ability to grow more effectively than other species when environmental conditions get more favorable has also been investigated in terms of mathematical models and statistical data analysis.
In  \cite{Schapira,SMETACEK199925}, it has been expressed by computing a growth rate  which is higher for diatoms than other phytoplankton types.
Based on this observation, Litchman \etal \cite{Litchman} have been interested in modeling a coefficient of nutrient  competitive uptake ability and estimated this coefficient for diatoms, coccolithophorids, dinoflagellates and green algae. 
By compiling a database of nutrient uptake and nutrient-dependent growth parameters, the authors  have fitted the parameters of a model based on quota-Droop's law \cite{droop1968} and  Michaelis-Menten-Monod \cite{turpin1988} equation. 
Their analysis showed the ability of diatoms to respond more effectively to the return of good environmental conditions. 

\subsection{Stoichiometry and Redfield ratio}
\label{sec:stoichiometry}
Redfield discovered in 1934 that the ratio of carbon, nitrogen and phosphorus (C:N:P) is  nearly constant in oceans, in phytoplankton biomass as well as in dissolved nutrients. 
Concerning diatoms, the Redfield stoichiometry is usually given by $C:N:P = 106:16:1$  \cite{Brzezinski}, even if it is known that it may fluctuate to become C:N:P = 163:22:1 \cite{Concentrationsandratiosofparticulate,Klausmeier}.
These values have to be understood as averages and may vary significantly locally. 
In particular, intracellular stoichiometry can strongly vary due to nutrient limitations \cite{Saudo-Wilhelmy}.
\\[0.4cm] \indent
Imbalance in stoichiometry can lead to many consequences. Algal growth rate can be impacted 
and abnormal secretions are induced as developed in \cite{Elser}: \textit{"When nutrient element X is in excess relative to requirements, consumers increase the specific rate of loss of that element (\textit{via} excretion or egestion)"}.
Resulting in particular in a synthesis and secretion of mucilages, as pointed out in the conclusion of \cite{Mecozzia}, "\textit{[...] the formation of mucilages is an alteration of the natural humification process that occurs when the organic matter degradation phase becomes slower than the synthesis phase. This alteration can [...] support the formation of refractory materials such as macro-aggregates and mucilages.}" 
Thus, together with a limitation of P stopping cellular division, 
stoichiometry imbalance  may be a hindrance to the development of bacteria  in charge of the degradation of this organic matter (mucilage).
The consequence is probably the development of  liga.

\subsubsection{Imbalance N:P}
The excess of nitrogen in relation to phosphates is pointed out by \cite{Susperregui} to be a possible trigger of liga production in a context of P limitation.
A prevailing influence of nitrogen is also mentioned in situation of mucilage in the Adriatic sea in \cite{BUZZELLI19971171},  in the Tasman Bay in \cite{MacKenzie}, where a similar imbalance is observed but interpreted as an effect rather than a cause of TEP production.
\\[0.4cm] \indent
In case of high level of N in comparaison to P, one is interested in polymeric secretions with a high content of N rather than P, whereas TEP have a composition poor in N.
Meanwhile, CSP 
\footnote{CSP: Coomassie Stainable Particles}
are N-rich exopolymer particles that have recently been studied in \cite{Thornton}:
the author notes that  the  C:N ratio  for CSP is 3.8:1 (lower than Redfield's ratio of 6.6:1), though it can be estimated to be 26:1  for TEP.
\\[0.4cm] \indent
If CSP can be understood as the excretion of an excess of N, it is remarked in \cite{Mari} that contrary to TEP, CSP do not seem to significantly impact aggregation processes and  may not play the same pivotal role in carbon cycling as TEP. 
In conclusion, abnormal TEP  secretion inducing liga formation may rather be searched in a context of P limitation and of  C:P or C:N imbalances.

\subsubsection{Imbalance C:P or C:N}
From a geological point of view,  marine diatoms have been at the front line of the regulation of atmospheric partial pressure of carbon dioxide  (pCO$_2$),
in particular by reducing atmospheric carbon dioxide levels \cite{Cermeno}. 
Accumulations of diatom-rich sediments extend back to the late Cretaceous.  
More recently, diatom-rich lake sediments are from the Eocene, with large deposits found in the Miocene \cite{diatomite}.
Harwood \etal \cite{CretaceousRecordsofDiatomEvolution} dated the emergence of planktonic diatom species at the end of the early cretaceous (Albian), a marine diatom extinction during the first part of the late cretaceous (Cenomanian/Turonian) and the beginning of pelagic diatom deposits at the early Campanian. 
However, these dates may be older \cite{marinediatomsinupperAlbianamber,Accelerateddiversification} and could be linked to local peaks of partial pressure of carbon dioxide identifiable during these periods \cite{CO2Cretaceous}.
Afterwards, diatom population growth declined during the Cenozoic era \cite{QuantifyingtheCenozoicmarinediatom} with a global decreasing of the partial pressure of carbon dioxide perturbed by local peaks \cite{ConvergentCenozoicCO2history}. 
\\[0.4cm] \indent
Oceans absorb huge quantities of inorganic carbon due to important anthropic emissions of CO$_{\text{2}}$ into the atmosphere, increasing the stoichiometry of carbon to nitrogen. 
This has implications for a variety of marine biological and biogeochemical processes, and underscores the importance of biologically driven feedbacks in the ocean to global climate change. 
It is pointed out in \cite{Riebesell}  that for the same uptake of inorganic nutrients, the production of organic matter increases significantly with higher CO$_{\text{2}}$ partial pressures, probably in the form of TEP.
\\[0.4cm] \indent
Many authors related strong correlations between TEP production by plankton assemblages dominated by diatoms and cyanobacteria, and CO$_{\text{2}}$ concentration \cite{Barcelos,Brembu,Danovaro,Heemann,Hein,Hessen,Mari,Schapira,Torstensson,Wei} with different interpretations of the usefulness of these secretions: the excretion of an excess to rebalance Redfield's stoichiometry, buoyancy control, grazing repellents or to attract bacteria in charge of mucus degradation. 
\\[0.4cm] \indent
Following \cite{Mari}, the theoretical  Redfield's ratio C:N:P=106:16:1 is strongly related to the equation of production of phytoplankton biomass given by,
\begin{align*}
  106\,\text{CO}_2 + 122\,\text{H}_2\,\text{O} + 16\,\text{HNO}_3 + \text{H}_3\text{PO}_4 \longrightarrow \text{(CH}_2\text{O)}_{106}\text{(NH}_3\text{)}_{16}\text{(H}_3\text{PO}_4\text{)} + 138\,\text{O}_2
\end{align*}
whereas, without particular link to Redfield's ratio, the equation of production of TEP\footnote{under cover of the acceptance of its chemical writing} is given by
\begin{align*}
  106\,\text{CO}_2 + 105\,\text{H}_2\text{O}  \longrightarrow \text{C}_{106}\,\text{H}_{201}\text{O}_{105}+ 106\,\text{O}_2.
\end{align*}
An imbalance C:P or C:N then could trigger TEP production by diatoms to rebalance their  Redfield's ratio. This has been observed in \cite{Mari}: 
under optimal growth conditions, elevated pCO$_2$ leads to increased TEP production by diatom  \textit{Skeletonema costatum}.
\\[0.4cm] \indent
Of course, reality is more complex.
It is in particular species dependent and in \cite{Mari} the same experiment for diatom  \textit{Thalassiosira weissflogii} did not change TEP production (which however increased at sub-optimal irradiance and growth temperature).
The variation of TEP production can have other origins: nutrient limitation (cellular carbon overflow),  temperature  (an increase of temperature can enhance extracellular releases), see \cite{Mari}. 
\section{Model statement}
\label{sec:math-mod}
  Models of marine ecosystem  require to consider a large number of variables
  such as different types of phytoplankton and of grazers, of nutrients, environmental conditions or photosynthesis models.
  Its vocation is to describe biological diversity over a rather large period of time. 
  We refer \eg  to \cite{Dutkiewicz,Follows,MathildeCadier} and \cite{FlynnMitra} for such classic models in plankton ecology, which however do not model mucilage production.  
\\[0.4cm]\indent
Our aim is to propose a mathematical model that describes the evolution of a phytoplankton population during a period of bloom and the corresponding evolution of TEP production.
Up to a prescribed threshold of TEP concentration, the first stage of liga (as described in introduction section) will be considered to be reached.  
Pelagic ecosystems being very complex, only the key features of our understanding of liga production will be kept, leading to an inevitable and drastic simplification of reality.
In the light of what has been said in previous sections, diatom  phytoplankton species will be considered to be at the origin of TEP production,  
in response to  an imbalance of C:N stoichiometry  during a bloom of population with P limitation. 
Predation will be neglected: during a diatom bloom, predators are expected to grow with a delay (ecological advantages of diatoms in terms of growth are developed in section \ref{Diatom-blooms}).
More generally let us emphasize again that the goal is to model the first stage leading to liga production and not to study the global ecosystem of liga. 
\\[0.4cm]\indent
Phytoplankton population models classically rely on \textit{quota}. 
The basic tools are: Michaelis-Menten-Monod rule (see \cite{turpin1988}) for the uptake of nutrients 
and  Droop-Liebig's law (see \cite{droop1968}) to specify the impact of quotas and limitation on the phytoplankton growth. 
\\
Let us briefly recall these concepts.
  For a given species with concentration $D$ and a given nutrient with concentration $N$, the  associated \textit{quota} is defined by $Q = \frac N D$. 
\\
  The quota-Droop's law models a net-growth rate by 
  $\mu_\infty \Big(1-\frac{Q_{\min}}{Q}\Big) - m$
  (where $Q_{\min}$ is the minimum quota, $\mu_\infty$ is the growth rate of species at infinite quota, $m$ is the mortality rate).
\\
  The Michaelis-Menten-Monod  equation defines a nutrient uptake rate by 
  $ V_{\max} \frac{N}{K+N}$ 
  (where $V_{\max}$ is the maximum nutrient uptake rate, $K$ is the half-saturation constant for nutrient uptake).
\subsection{Discussion of some mathematical models}
\label{sec:model-intro}
In this section are presented a selection of plankton models found in the literature.
Without attempting to be comprehensive, we want to see how to adapt these models to go towards liga modeling.
\\
In \cite{Legovic}, a phytoplankton  growth model relying on multiple nutrients is proposed, including nutrient limitation and stoichiometry information. 
In \cite{Baird} a similar model is inserted in a more complex setting involving diffusion/convection, light, sinking rates and predators. 
\\
The mathematical model in \cite{Legovic} has been confronted 
to experimental data from chemostat\footnote{chemostat: a system for the culture of microorganisms
with controlled chemical composition.}
experiments in \cite{Legovic}.
The authors concluded that their results are relatively satisfactory, except for  phytoplankton stoichiometry during exponential growth at high nutrient levels.
Two reasons can be raised for that. Firstly, no process of extra-cellular secretion is modeled, which would allow to eliminate a nutrient excess. 
Secondly,  
in the quota equilibrium equation, no consumption negative term for growth is present excepted  cell-division.
\\
In \cite{Peace}, a predator prey model is proposed including  one algal and one grazer species and a nutrient in excess. 
A waste of up-taken  nutrient not used for growth is added, 
as in  \cite{Baird}, 
with a feedback from the producer biomass onto nutrient concentration.
Meanwhile, no quota-decrease by cell division is taken into account in the model. 
It is not obvious at all that an excess of up-taken nutrient may return as available nutrients as proposed by  \cite{Baird,Peace}.
Consider for instance carbon: it is up-taken by the phytoplankton as inorganic carbon and then transformed, after photosynthesis, into polysaccharides and cannot be secreted back to the extra-cellular media as inorganic carbon to be up-taken again.
Alternatively, 
being born in mind the importance of stoichiometry imbalances previously discussed in section \ref{sec:stoichiometry}, 
it seems necessary to insert in the model 
a possible secretion of  a given nutrient in excess, especially when another nutrient is limited.
This certainly helps to explain the problem of phytoplankton stoichiometry during exponential growth in \cite{KlausmeierLitchman}.
\\[0.2cm]\indent
Thus, it appears as  necessary to add to the model one or more additional variables corresponding to secretion of EPS (Extracellular polymeric substances) 
when given nutrients are in surplus: TEP concerning C or CSP for N and C. 
\subsection{A  two nutrient model}
A chemostat experiment is considered  with an inflow/outflow flushing rate given by $a>0$.  
Diatom concentration (the biomass producer in our case) 
is denoted $D$.
Two nutrients are considered, nitrogen and carbon, the associated  concentration are  
$N$ and $C$ in the extra cellular media, and  
$P_N$ and $P_C$ in the phytoplankton.
The associated quota are $Q_N=\frac{P_N }{D}$, $Q_C=\frac{P_C}{ D}$ and
$M$ will denote EPS concentration.
\\[0.4cm]\indent
Let the chemostat inlet have a nitrogen concentration 
$\Nin $, the inflow/outflow nitrogen balance is $a(\Nin -N)$.
The diatom uptake rate in nitrogen is denoted $\tauN$ and in the absence of other species in the chemostat, the evolution law for nutrient $N$ is,
\begin{align*}
  \diff{N}{t} = a(\Nin -N) - \tauN D
\end{align*}
and similarly for $C$,
\begin{align*}
  \diff{C}{t} = a(\Cin -C) - \tauC D
\end{align*}
Denoting $\sigmaN$ the consumption rate of $N$ by diatoms (\textit{e.g.} for growth) and by $m_D$ the death-rate of diatoms, 
the concentration of $N$ in the diatoms is ruled by
\begin{align*}
  \diff{P_N}{t} = (\tauN-\sigmaN) D - (a+m_{D}) P_N, 
\end{align*}
and similarly for $C$,
\begin{align*}
  \diff{P_C}{t} = (\tauC-\sigmaC) D - (a+m_{D}) P_C.
\end{align*}
The concentration in diatoms will be given by,
\begin{align*}
  \diff{D}{t} = \tauD D - (a+m_D) D,
\end{align*}
where $\tauD$ is the growth rate of the diatom population. 
Finally, mucilage concentration satisfies,
\begin{align*}
  \diff{M}{t} = \tauM D - a M,
\end{align*}
where $\tauM$ is the secretion rate of EPS. 
\\
In terms of quota, differentiating their definition 
$\frac{\d Q_N}{\dt} = \frac{1}{D}\,\frac{\d P_N}{ \dt} \;-\, \frac{P_N}{D^2} \, \frac{\d D}{\dt}$,
\begin{align*}
  \diff{ Q_N}{t}  =&\tauN-\sigmaN - Q_N\tauD,
  \\
  \diff{ Q_C}{t} =& \tauC-\sigmaC - Q_C\tauD,
\end{align*} 
The uptake rate is given by Michaelis-Menten-Monod law as introduced in section \ref{sec:model-intro},
\begin{align}
  \label{tau_N}
  \tauN &= \tau_N(t,N) = V^N_{\max}(t) \frac{N}{K_N(t)+N},
  \\
  \label{tau_C}
  \tauC &= \tau_C(t,C) = V^C_{\max}(t) \frac{C}{K_C(t)+C}.
\end{align}
\\[0.2cm]\indent
  In the sequel $x^+$ denotes the positive part of the real number $x$: $x^+=\max(0, x)$.
Cell division can take place if the quota of each nutrient is greater than a threshold: $Q_C \geq \QCMin>0$ and $Q_N \geq \QNMin>0$. Then, following \cite{Legovic}, Droop's growth function coupled with Liebig's law yields
\begin{align}
  \label{tau_D}
  \tauD = \tau_D(t,Q_C,Q_N)
  = \mu_D(t) 
  \min\left (
  (1-\dfrac{\QCMin}{Q_C})^+
  ,~
  (1-\dfrac{\QNMin}{Q_N})^+
  \right )
\end{align}
where $\mu_D$ is the maximal division rate at infinite quota.  
In \textit{normal conditions}, the diatom stoichiometry ratio $C:N$ is a positive constant denoted $\alpha$: diatoms will use $\alpha $ units of $C$ for one unit of $N$. 
Considering that EPS release is based on the presence of extra $C$ compared to $N$, \textit{i.e.} if $Q_C > \alpha Q_N$ and, of course, if $Q_C>\QCMin$, the EPS release rate $\tauM$ will be modeled by,
\begin{align}
  \label{tau_M}
  \tauM = \tau_M(t,Q_C,Q_N) = 
  \mu_M(t) 
  \min\left (
  (1-\dfrac{\QCMin}{Q_C})^+
  ,~
  (1-\dfrac{\alpha Q_N}{Q_C})^+
  \right ),
\end{align}
with $\mu_M$  the maximal release rate of mucilage by diatoms.
\\[0.2cm]
Concerning nutrient consumption, 
\begin{align}
  \label{sigma_N}
  \sigmaN = \sigmaN(t,Q_C,Q_N) = 
  \Theta_D(t) 
  \min\left (
  (1-\dfrac{\QCMin}{Q_C})^+
  ,~
  (1-\dfrac{\QNMin}{Q_N})^+
  \right ),
\end{align}
where $\Theta_D$ the  maximal consumption rate of diatoms for growth at infinite quota. 
\\
Biomass production for the growth of diatom population  
respects the stoichiometry $\alpha $ units of $C$ for one unit of $N$,
thus the carbon consumption rate is split into two parts, one for diatom growth and the second for mucilage production,
\begin{displaymath}
  \sigmaC = \alpha\,\sigmaN + \sigmaM
\end{displaymath}
with,
\begin{align}
  \label{sigma_M}
  \sigmaM =  \sigmaM(t,Q_C,Q_N) =
  \Theta_M(t)
  \min\left (
  (1-\dfrac{\QCMin}{Q_C})^+
  ,~
  (1-\dfrac{\alpha Q_N}{Q_C})^+
  \right ),
\end{align}
where $\Theta_M$ is the maximal consumption rate of diatoms for mucilage.
\\[0.2cm]\indent
In short, we have the system:
\begin{equation}
  \label{eq:system-S} \tag{S}
  \frac{\d X}{\d t}=F(t,X),
\end{equation}
where the model variable have been gathered in the column vector 
$X= (N, C, Q_N, Q_C, D,M)^T$ and where the function $F$ summarizes as,
\begin{align*}
  \diff{N}{t} =& a(\Nin(t) -N) - \tauN(t,N) D ~:=~F_1(t,X),
  \\[0.1cm]
  \diff{C}{t} =& a(\Cin(t) -C) - \tauC(t,C) D~:=~F_2(t,X),
  \\[0.1cm]
  \diff{ Q_N}{t}  =&  \tauN(t,N)-\sigmaN(t,Q_C,Q_N)  - \tauD(t,Q_C,Q_N)Q_N~:=~F_3(t,X),
  \\[0.1cm]
  \diff{ Q_C}{t} =&  \tauC(t,C)- \alpha\, \sigmaN(t,Q_C,Q_N) - \sigmaM(t,Q_C,Q_N) - \tauD(t,Q_C,Q_N)Q_C~:=~F_4(t,X),
  \\[0.1cm]
  \diff{D}{t} =& \tauD(t,Q_C,Q_N) D - (a+m_D(t)) D~:=~F_5(t,X),
  \\[0.1cm]
  \diff{M}{t} =& \tauM(t,Q_C,Q_N) D - a M~:=~F_6(t,X).
\end{align*}
Functions
$\tau_N$ ,$\tau_C$, $\tau_D$, $\tau_M$, 
$\sigma_N$ and $\sigma_C$ 
(given by 
(\ref{tau_N}) to (\ref{sigma_M}) respectively), 
depend on  $K_N(t)$, $K_C(t)$,  $V_{\rm max}^N(t)$, $V_{\rm max}^C(t)$, $\mu_D(t)$, $\mu_M(t)$, $\Theta_D(t)$ and $\Theta_M(t)$.
  These parameters do \textit{a priori} depend on the time $t$ (since they are functions of lightness and temperature, see \textit{e.g.} \cite{Dauta}), although they will be set to constants in numerical experiments.
The same remark holds for $\Nin(t)$, $\Cin(t)$ and $m_D(t)$.
The model finally also depends on the constant parameters
$a$, $\QNMin$, $\QCMin$ and $\alpha$.
\section{Model mathematical analysis}
For technical reasons, but fully compatible with biological modeling, the following assumptions are made on the parameters of system \eqref{eq:system-S},
\begin{align}
  \label{bornes}
  &\exists \mathcal{M}>0, \quad a, \QCMin, \QNMin, \alpha \in (\frac1{\mathcal{M}},\mathcal{M}), \quad m_D\in C\Big([0,+\infty),(0,\mathcal{M})\Big),
  \\  \label{bornes2}&
                       \Nin , \Cin , V^N_{\max}, K_N, V^C_{\max}, K_C, \mu_D, \mu_M, \Theta_D, \Theta_M  \in C\Big([0,+\infty),(\frac1{\mathcal{M}},\mathcal{M})\Big).
\end{align}
On the contrary of the other parameter functions, the mortality rate $m_D$ is allowed to vanish, which is relevant in the present context of a  chemostat-type experiment.
\\[0.2cm] \indent
Let us fix what we mean by a solution to the problem. 
\begin{definition}\label{def_solution}
A solution to system \eqref{eq:system-S} is any absolutely continuous function $X$ with non negative values, solution to the ordinary differential equation $\frac{\d X}{\d t}=F(t,X)$ associated with the initial condition $X(0)=X_0$ where $X_0\geq 0$.
\\
For convenience, for $X\in\R^6$, $X \ge 0$ means that $X_i\ge 0$ for $1\le i\le 6$.
\end{definition}

Our main result is the following.  
\begin{theorem}
  \label{theo:main-result}
Under assumptions \eqref{bornes} and \eqref{bornes2}, 
\begin{enumerate}
\item there exists a unique maximal solution to system \eqref{eq:system-S} in the sense of definition \ref{def_solution}, 
\item this unique solution is a global solution (defined for $t\in |0,+\infty)$),
\item if there exists a time $t_0\geq 0$ such that $X_3(t_0) = Q_{N}(t_0) \ge \QNMin$ then $Q_{N}(t) \ge \QNMin$ for any $t\geq t_0$. 
Similarly, if $X_4(t_0) = Q_{C}(t_0) \ge \QCMin$ then $ Q_{C}(t) \ge \QCMin$ for any $t\geq t_0$ 
(in other words, the solution remains \textit{biologically relevant}).
\end{enumerate}
Moreover, in the case of constant parameters, 
the following regularity with respect to the parameters holds.
Denote 
\begin{align*}
  A=(a,\Nin,\Cin,V^N_{\max}, K_N, V^C_{\max}, K_C, \QCMin, \QNMin, \alpha, \mu_D, \mu_M, \Theta_D, \Theta_M, m_D),
\end{align*}
a vector of constant parameters  in $\mathcal{A}=(\frac1{\mathcal{M}},\mathcal{M})^{15}$. 
Consider a given initial condition $X_0 \ge 0$ to problem (\ref{eq:system-S+}).
\begin{enumerate} \setcounter{enumi}{3}
\item 
For any $T>0$, 
the unique solution $X(t,A)$, 
as a function of the time variable  $t\in [0,T]$ 
and of the parameters $A \in \overline{\mathcal{A}}$ 
is weakly differentiable,
$$X\in W^{1,\infty}\left((0,T)\times \mathcal{A}\right)^6,$$
and $t \mapsto X(t,\cdot) \in C^0(\overline{\mathcal{A}})^6$ is continuously differentiable in $t$,
$$X\in C^1([0,T],C^0(\overline{\mathcal{A}})^6),$$ 
and,  
$$X\in {\rm Lip}([0,T],W^{1,\infty}(\mathcal{A})^6).$$ 
\\
$\nabla_A X$ is the unique solution to the linear ordinary differential equation obtained by differentiation with respect to $A$ of system \eqref{eq:system-S},
\begin{displaymath}
\frac{\d}{\d t} \nabla_A X(t,A) = \dfrac{\partial F}{\partial  X}
\big(X(t,A),A\big)
\nabla_A X(t,A)
 ~+~ \dfrac{\partial F}{\partial  A}\big(X(t,A),A\big),  
\end{displaymath}
with the initial condition $(\nabla_A X)_0 = 0$.
\end{enumerate}
\end{theorem}
The proof of Theorem \ref{theo:main-result} has been split
into lemma \ref{lem:ineq-F}, \ref{positivite} \& \ref{globale} for points 1 \& 2  and
lemma \ref{seuil} for point 3. Then, point 4 is detailed in lemmas 
\ref{lem:X-weak-differentiable}, \ref{lem:stablilite} and \ref{lem:differentiabilite}. 
\\[0.1cm]
Finally, in the last section \ref{Numerical_simulations}, 
the differentiability properties enunciated in  theorem 
\ref{theo:main-result}
will be used to propose a method of parameter identification for model  (\ref{eq:system-S+}) based on a Gauss Newton method: 
\begin{itemize}
\item a realistic choice of parameters is discussed, illustrated by  direct simulations showing the model dynamics,
\item then, a numerical methodology for identifying the problem parameters is presented, supported by numerical output statistics.
\end{itemize}
\subsection{Well-posedness}
To address the question of the positivity of the solutions of 
system \eqref{eq:system-S}, we study the Cauchy problem,
\begin{equation}
  \label{eq:system-S+}
  \diff{X}{t}=F(t, X^+)\,, \quad X(0)=X_0.
\end{equation}
It will be proven in this section that problem \eqref{eq:system-S+} is well posed with global solutions for all $t\ge 0$.
Moreover its solutions will satisfy $X(t)\ge 0$ if $X_0\ge 0$ and so will coincide with system \eqref{eq:system-S} solutions.
\\[0.2cm] \indent
In the sequel, we generically denote $C>0$ a positive constant only depending on the parameter $\mathcal{M}$ in assumptions (\ref{bornes})~(\ref{bornes2}), whose value however may differ from one situation to another.
\\[0.1cm]
Let us note that for positive $\QCMin$ and $\QNMin$, and nonnegative $Q_C$ and $Q_N$,  
{\footnotesize \begin{align*}
\min\left ((1-\dfrac{\QCMin}{Q_C})^+,~
(1-\dfrac{\QNMin}{Q_N})^+\right ) 
=
\min\left (1-\dfrac{\QCMin}{\max(\QCMin,Q_C)},~
1-\dfrac{\QNMin}{\max(\QNMin,Q_N)}\right ),  
\end{align*}}
\noindent and, noticing that $(1-\dfrac{\QCMin}{Q_C})^+=0$ for any $0 \leq Q_C\leq  \QCMin$, that
{\footnotesize \begin{align*}
\min\left ((1-\dfrac{\QCMin}{Q_C})^+,~
(1-\dfrac{\alpha Q_N}{Q_C})^+\right ) 
=
\min\left (1-\dfrac{\QCMin}{\max(\QCMin,Q_C)},~
1-\dfrac{\alpha Q_N}{\max(\alpha Q_N,Q_C)}\right ),  
\end{align*}}
\noindent with the convention that $\dfrac{\alpha Q_N}{\max(\alpha Q_N,Q_C)}=0$ if  $Q_N= Q_C= 0$. 
Then, the following function $\A$, defined \textit{a priori} for any  $(a,b,c,d) \in \R^4$ by 
\begin{equation}
  \label{def:FonctionA}
  \A(a,b,c,d) = 
      \min
      \Big(   
      1-\dfrac{a^+}{\max(a^+,b^+)},
      1-\dfrac{c^+}{\max(c^+,d^+)}
      \Big)      
~~\text{with convention }~~\frac00=0,
\end{equation}
is introduced for the study of problem \eqref{eq:system-S+}, it satisfies the following lemma.
\begin{lemma}
  \label{lem:FonctionA}
  \begin{itemize}
\item 
For any $(a,b,c,d) \in \R^4$, $0 \le \A(a,b,c,d) \le 1$.  
  \item Let $\delta>0$, $(a_i,b_i,c_i,d_i) \in \R^4$ with $a_i,c_i\geq \delta$, $i=1,2$, then, 
  \begin{align*}
\Big|\A(a_1,b_1,c_1,d_1)-\A(a_2,b_2,c_2,d_2)\Big| 
\leq 2\frac{|a_1-a_2|+|c_1-c_2|}{\delta}+\frac{|b_1-b_2|+|d_1-d_2|}{\delta}.
\end{align*}
 
  \item
 Let $\delta>0$, $(a_i,b_i,c_i) \in \R^4$ with $a_i\geq \delta$, $i=1,2$, then, 
  \begin{align*}
\Big|\A(a_1,b_1,c_1,b_1)-\A(a_2,b_2,c_2,b_2)\Big| 
\leq 2\frac{|a_1-a_2|+|c_1-c_2|+|b_1-b_2|}{\delta}.
\end{align*} 
  \end{itemize}
\end{lemma}
\begin{proof}
Consider $(a_i,b_i,c_i,d_i) \in [\delta,+\infty[\times \R^+\times [\delta,+\infty[\times \R^+$ for a given positive $\delta$, $i=1$ and $2$. 
Using the following inequality on the $\min$ and $\max$ functions,
$\vert \min(x,y) - \min(a,b)   \vert \le \vert x-a\vert  + \vert y-b\vert$ 
and 
$\vert \max(x,y) - \max(a,b)   \vert \le \vert x-a\vert  + \vert y-b\vert $,   one has 
\begin{align*}
&\Big|\A(a_1,b_1,c_1,d_1)-\A(a_2,b_2,c_2,d_2)\Big| 
\\=& \Big|\min\Big(1-\dfrac{a_1}{\max(a_1,b_1)},1-\dfrac{c_1}{\max(c_1,d_1)}\Big)-\min\Big(1-\dfrac{a_2}{\max(a_2,b_2)},1-\dfrac{c_2}{\max(c_2,d_2)}\Big)\Big|
\\ \leq &
\Big|\dfrac{a_1}{\max(a_1,b_1)}-\dfrac{a_2}{\max(a_2,b_2)}\Big|+ \Big|\dfrac{c_1}{\max(c_1,d_1)}-\dfrac{c_2}{\max(c_2,d_2)}\Big|
\\ \leq &
\Big|\dfrac{(a_1-a_2)\max(a_2,b_2) + a_2[\max(a_2,b_2)- \max(a_1,b_1)]}{\max(a_1,b_1)\max(a_2,b_2)}\Big|+ \Big|\dfrac{c_1}{\max(c_1,d_1)}-\dfrac{c_2}{\max(c_2,d_2)}\Big|
\\ \leq &
\dfrac{|a_1-a_2|}{\max(a_1,b_1)}+
\dfrac{|\max(a_2,b_2)- \max(a_1,b_1)|}{\max(a_1,b_1)}+ 
\dfrac{|c_1-c_2|}{\max(c_1,d_1)}+
\dfrac{|\max(c_2,d_2)- \max(c_1,d_1)|}{\max(c_1,d_1)}
\\ \leq &\frac2\delta \Big[|a_1-a_2|+|c_1-c_2|\Big]+\frac1\delta \Big[|b_1-b_2|+|d_1-d_2|\Big].
\end{align*}
In case of signed variables, the inequality holds with $a^+_i,b^+_i,c^+_i,d^+_i$ (i=1,2) on the right hand side and the result holds since $|x^+-y^+| \leq |x-y|$ for any real $x,y$. 
\\[0.1cm]
Now consider the second inequality, first for non negative variables.
It is here assumed that $a_i\ge \delta$.
In the particular case where $b_i > a_i$, one has $b_1,b_2\ge \delta$ and the second inequality is induced by the first one.
Now if $b_i \le a_i$ then $ \A(a_i,b_i,c_i,b_i) = 0 $,  we therefore only have 
to investigate the case $b_1 \leq a_1$ and $b_2 > a_2$ (the fourth case being similar), 
\begin{align*}
\Big|\A(a_1,b_1,c_1,b_1)-\A(a_2,b_2,c_2,b_2)\Big| =
\A(a_2,b_2,c_2,b_2) \leq 1 - \dfrac{a_2}{b_2},
\end{align*}
and, since $b_1-a_1 \le 0$,
\begin{align*}
1 - \dfrac{a_2}{b_2} = 
\dfrac{b_2-b_1 + (b_1-a_1) + a_1-a_2 }{b_2}
\leq 
\dfrac{|b_2-b_1| + |a_1-a_2|}{\delta}.
\end{align*}

The case of signed variables is similar as above.
\end{proof}
The well-posedness of the Cauchy problem \eqref{eq:system-S+} is stated in the following lemma.
\begin{lemma}
  \label{lem:ineq-F}
  With assumptions \eqref{bornes} and \eqref{bornes2}, 
  the function
  $(t,X)\mapsto F(t, X^+)$ is continuous  and locally Lipschitz
  with respect to its second variable $X$ on $[0,+\infty)\times  \R^6$. 

  As a consequence, Picard–Lindel\"of theorem yields 
a result of existence and uniqueness of the maximal solution 
  to the Cauchy problem \eqref{eq:system-S+},
  for any initial condition $X_0\in \R^6$.
\end{lemma}
\begin{proof}
  It will be denoted here $F_i^+:(t,X)\mapsto F_i(t, X^+)$.
  \begin{align*}
    F_1(t,X_1^+)-F_1(t,X_2^+)&=  a N_2^+- a N_1^+ 
                               + \tau_N(t,N_2^+)D_2^+ - \tau_N(t,N_1^+)D_1^+    
    \\&=
    a(N_2^+- N_1^+ ) + D_2^+ (\tau_N(t,N_2^+) - \tau_N(t,N_1^+))
    \\& ~~+ \tau_N(t,N_1^+)(D_2^+ - D_1^+).
  \end{align*}  
  Thanks to assumptions (\ref{bornes}), $\tau_N(t,N^+)$ in (\ref{tau_N}) is bounded and globally Lipschitz in $N$ on 
  $[0,+\infty)\times [0,+\infty)$ 
  with a constant only depending on $\mathcal{M}$. Then,
  \begin{align*}
    \left \vert  
    F_1(t,X_1^+)-F_1(t,X_2^+)
    \right  \vert  \le
    C(1 + |D_2|) \left (
    |N_1-  N_2| + |D_1-  D_2|
    \right ),
  \end{align*}  
  and $F_1^+$ is continuous and locally Lipschitz in $X$ on $[0,+\infty)\times \R^6$, as well as $F_2^+$ and $F_5^+$ with a similar proof.
  \\[0.4cm]\indent
  Let us now consider  $F_3^+:(t,X)\mapsto F_3(t, X^+)$, we recall that
  $
  F_3(t,X^+) = \tau_N(t,N^+) - \sigma_N(t,Q_C^+,Q_N^+)
  - \tau_D(t,Q_C^+,Q_N^+)Q_N^+ $.
  We just stated that the first term is globally Lipschitz in $N$.
  With definition  \eqref{def:FonctionA} of  $\A$ we get,
  \begin{align*}
    \sigma_N(t,Q_C^+,Q_N^+) &= 
                              \Theta_D(t)\A(\QCMin, Q_C,\QNMin, Q_N),
    \\ 
    \tau_D(t,Q_C^+,Q_N^+)   &= 
                              \mu_D(t)\A(\QCMin, Q_C,\QNMin, Q_N).
  \end{align*}
  With assumption (\ref{bornes2}), $\Theta_D$ and $\mu_D$ are bounded
  and with lemma \ref{lem:FonctionA}:
  $\A(\QCMin, Q_{C},\QNMin, Q_{N}) \le 1$ and,
  \begin{align*}
    \left \vert 
    \A(\QCMin, Q_{C,1},\QNMin, Q_{N,1})
    -
    \A(\QCMin, Q_{C,2},\QNMin, Q_{N,2})
    \right \vert 
    \\
    \le
    C  
    \left (|Q_{C,1}-Q_{C,2}|+|Q_{N,1}-Q_{N,2}|\right ).
  \end{align*}
  It follows that,
  \begin{align*}
    \left \vert  
    \sigma_N(t,Q_{C,1},Q_{N,1}) - 
    \sigma_N(t,(Q_{C,2},Q_{N,2}) 
    \right\vert  
    \leq C  
    \left (|Q_{C,1}-Q_{C,2}|+|Q_{N,1}-Q_{N,2}|\right ),
  \end{align*}
  and,
  \begin{align*}
    &\left \vert 
      \tau_D(t,Q_{C,1}^+,Q_{N,1}^+)Q^+_{N,1} - 
      \tau_D(t,Q_{C,2}^+,Q_{N_2}^+)Q^+_{N,2}
      \right \vert 
    \\
    &~~~\le
      \vert \tau_D(t,Q_{C,1}^+,Q_{N,1}^+) \vert  \, 
      \vert Q^+_{N,1} - Q^+_{N,2}\vert 
      + 
      \left \vert \tau_D(t,Q_{C,1}^+,Q_{N,1}^+) - \tau_D(t,Q_{C,2}^+,Q_{N_2}^+)
      \right \vert \,\vert Q^+_{N,1}\vert  
    \\
    &~~~\le
      C
      \left \vert Q_{N,1} - Q_{N,2} \right \vert 
      +
      C  
      \left (|Q_{C,1}-Q_{C,2}|+|Q_{N,1}-Q_{N,2}|\right )
      \vert Q_{N,1}\vert.  
  \end{align*}
  Altogether, this prove that $F_3^+$ is continuous and locally Lipschitz in $X$ on $[0,+\infty)\times \R^6$.
  \\[0.4cm]\indent
  We now consider  $F_4^+$,
  we recall that
  $
  F_4^+(t,X) = g_4(t,X) - \sigma_M(t,Q_C^+,Q_N^+) $
\\  with \ 
$g_4(t,X^+) := \tau_C(t,C^+) - \alpha \sigma_N(t,Q_C^+,Q_N^+)
  - \tau_D(t,Q_C^+,Q_N^+)Q_C^+$. 
\\
  From what has been previously developed , $g_4$ is locally Lipschitz in $X$, 
  and it remains to study the term,
  \begin{align*}
    \sigma_M(t,Q_C^+,Q_N^+) = \Theta_M(t)
    \A(\QCMin, Q_C, \alpha Q_N^+, Q_C).
  \end{align*}
As $\Theta_M$ is bounded with assumption (\ref{bornes2}), this latter is Lipschitz continuous in $X$ thanks to the last point of  lemma \ref{lem:FonctionA}.
So, $F_4^+$ is continuous and locally Lipschitz in $X$ on $[0,+\infty)\times \R^6$, as well as $F_6^+$ with a similar proof.
\end{proof}
Let us now prove positivity properties for the solutions of problem (\ref{eq:system-S+}): it implies the equivalence between formulations \eqref{eq:system-S} and (\ref{eq:system-S+}) for non-negative initial conditions $X_0$.
\begin{lemma}\label{positivite}
  With assumptions \eqref{bornes} and \eqref{bornes2}, 
  the unique maximal solution $X$ to problem (\ref{eq:system-S+}) 
  for an initial condition $X_0 \ge 0$ 
  remains non-negative: $X(t) \ge 0$ as long as it exists.
  \\
  This is the unique solution of system \eqref{eq:system-S} in the sense of definition \ref{def_solution}.
\end{lemma}
\begin{proof}
  Consider a solution $X$ to problem (\ref{eq:system-S+}) for $X_0\ge 0$. 
  Let us prove that for any of its coordinate $i$,  $X^-_i(t) = 0$ where $X^-_i=\max(0,-X_i)$ is the negative part of $X_i$. 
  Thanks to the chain rule in Sobolev spaces 
  $\dfrac{\d X_i^-}{ \d t} = 
  -\dfrac{\d X_i}{ \d t}\; 1_{\{X_i\leq 0\}} 
  = - F_i(\cdot, X^+) \; 1_{\{X_i\leq 0\}}$ for a.e. time $t$.
  Now the key argument is to notice in system (\ref{eq:system-S}) that
  $ F_i(\cdot, X^+) \ge 0$ if $X_i \le 0$.
  It follows that $\frac{\d X_i^-}{\d t} \le 0$, 
  $X_i(0)^-=0$ by assumption and so $X_i(t)^-=0$.
\end{proof}
We now investigate global existence for solutions to problem (\ref{eq:system-S+}).
\begin{lemma}\label{globale}
  With assumptions \eqref{bornes} and \eqref{bornes2}, 
the maximal solution to problem (\ref{eq:system-S+}) 
is a global solution.
\end{lemma}
\begin{proof}
The Euclidian norm on $\R^n$ is denoted $\Vert \cdot\Vert $.
Consider a solution $X(t)$ to problem (\ref{eq:system-S}) for an initial condition $X_0\ge 0$:

$$\dfrac{1}{2} \diff{\Vert X(t)\Vert^2}{t} = \diff{ X(t) }{t} \cdot X(t) 
= F(t,X(t))\cdot X(t),
$$
and, by a sign argument (since $X(t)\ge 0$ with lemma \ref{positivite}),  
\begin{align*}
  &F(t,X(t))\cdot X(t) \leq 
    a \Nin N + a \Cin C + V^N_{\max} \dfrac{NQ_N}{K_N+N} + V^C_{\max} \dfrac{CQ_C}{K_C+C}  \nonumber
  \\ &+ \mu_D \A(\QCMin,Q_C,\QNMin,Q_N) D^2  +\mu_M \A(\QCMin,Q_C,\alpha Q_N,Q_C)DM  \nonumber
  \\ \leq & 
            a \|\Nin\|_\infty^2 + aN^2 + a\|\Cin\|_\infty^2 +  aC^2 + \frac{\|V^N_{\max}\|_\infty}{\min_t K_N(t)} NQ_N + \frac{\|V^C_{\max}\|_\infty}{\min_t K_C(t)} CQ_C \nonumber
  \\ &+ \|\mu_D\|_\infty D^2  +\|\mu_M\|_\infty DM  \nonumber
       \leq 
       \dfrac{1}{2}C \left ( 1  +  \Vert X(t)\Vert^2 \right ),
\end{align*}
where $C$ is a constant depending on $\mathcal{M}$ in assumptions (\ref{bornes})~(\ref{bornes2}). 
Then  Gronwall lemma yields,
\begin{align}\label{Gronwall}
  |X(t)|^2 \leq e^{C t} 
  \left (
  |X_0|^2 + Ct \right ).
\end{align}
Assume that the maximal solution $X(t)$ is not a global one. Thus, there exists $T>0$ such that $\lim_{t \to T^-}||X(t)||=+\infty$, which contradicts the above inequality. 
\end{proof}
\subsection{Biologically relevant solutions} 
With lemma \ref{positivite} and \ref{globale}, system (\ref{eq:system-S}) 
associated with a non-negative initial condition has a unique solution which is non-negative and global.
It is stated here that the quota remain larger than the minimum quota if this is true at initial time, \textit{i.e.} biologically relevant initial condition are associated with biologically relevant solutions.
\begin{lemma}\label{seuil}
   With assumptions \eqref{bornes} and \eqref{bornes2}, 
   consider the solution to problem (\ref{eq:system-S}) 
   for an initial condition $X_0 \ge 0$.
   \\
   If there exists $t_0$ such that  $Q_N(t_0)\geq \QNMin$, then $Q_N(t)\geq \QNMin$ for any $t\geq t_0$.
   \\
   Similarly, if there exists $t_0$ such that  $Q_C(t_0)\geq \QCMin$, then $Q_C(t)\geq \QCMin$ for any $t\geq t_0$.
\end{lemma}
\begin{rem}
  Note that 
  \\
  i) if $Q_N \leq \QNMin$ in $(0,T)$ then $Q_N(t) = Q_N(0) + \displaystyle\int_0^t  \dfrac{V^N_{\max}(s)N(s)}{K_N(s)+N(s)}\ds$ if $t \leq T$.
  \\
  ii) if $Q_C \leq \QCMin$ in $(0,T)$ then $Q_C(t) = Q_C(0) + \displaystyle\int_0^t  \dfrac{V^C_{\max}(s)C(s)}{K_C(s)+C(s)}\ds$ if $t \leq T$. 
  \\
  iii) if $Q_N \leq \QNMin$  or $Q_C \leq \QCMin$ in $(0,T)$ then
  $D(t) = D(0)e^{-(a+m_D)t}$ in $(0,T)$.
  \\
  Therefore, in case of $N$ or $C$ limitation, 
  the diatoms are no longer able to divide and 
  diatom  population decreases and goes to 0 
  as long as this situations continues.
\end{rem} 
\begin{proof}
  The proof is similar as for lemma \ref{positivite}.
  Consider the component $X_3 = Q_N$ of $X(t)$,
\begin{align*}
  &\diff{ (Q_N-\QNMin)^-}{t}  = - 1_{\{Q_N<\QNMin\}}\diff{ Q_N}{t},
  \\ =& - 1_{\{Q_N<\QNMin\}}\Big[ V^N_{\max}(t) \dfrac{N}{K_N(t)+N}
        -[\Theta_D(t)+\mu_D(t) Q_N] \A(\QCMin,Q_C,\QNMin,Q_N)   \Big]
  \\ =& - 1_{\{Q_N<\QNMin\}} V^N_{\max}(t) \dfrac{N}{K_N(t)+N} \leq 0.
\end{align*}
Assume that $Q_N(t_0) \ge \QNMin $ for $t_0\ge 0$. 
since $(Q_N-\QNMin)^-$ is 
absolutely continuous, the integration of the above inequality for  $ t\geq t_0$ yields, 
$$(Q_N-\QNMin)^-(t) \leq (Q_N-\QNMin)^-(t_0)=0,$$
and so $Q_N(t)\geq \QNMin$.
Similarly,
\begin{align*}
  &\diff{ (Q_C-\QCMin)^-}{t} =- 1_{\{Q_C<\QCMin\}}\diff{ Q_C}{t}
  \\=& - 1_{\{Q_C<\QCMin\}} \Big[ V^C_{\max}(t) \dfrac{C}{K_C(t)+C} -[\alpha \Theta_D(t)+\mu_D(t) Q_C] \A(\QCMin,Q_C,\QNMin,Q_N)  
  \\ &
       \hspace{7.6cm}- \Theta_M(t) \A(\QCMin,Q_C,\alpha Q^+_N,Q_C)\Big]
  \\[-0.5cm]
  =& - 1_{\{Q_C<\QCMin\}} V^C_{\max}(t) \dfrac{C}{K_C(t)+C} \leq 0,
\end{align*}
which leads to the same result.
\end{proof}
\subsection{Stability with respect to parameters} 
This section deals with the stability of system (\ref{eq:system-S})
solutions with respect to its parameters.
The case of constant parameters is studied on a fixed time interval $[0,T]$.
As in theorem \ref{theo:main-result}, the set of parameters 
is denoted by a vector $A\in \mathcal{A}=(\frac1{\mathcal{M}},\mathcal{M})^{15}$, with $\mathcal{M}$ the bound in assumptions (\ref{bornes})~(\ref{bornes2}),
\begin{align*}
  A=(a,\Nin,\Cin,V^N_{\max}, K_N, V^C_{\max}, K_C, \QCMin, \QNMin, \alpha, \mu_D, \mu_M, \Theta_D, \Theta_M, m_D) \in \mathcal{A}.
\end{align*}
Through this section,
\begin{itemize}
\item 
an initial condition $X_0\ge 0$ and a final time $T>0$ are fixed,
\item 
$F:(Y,A)\in \R^6 \times \mathcal{A} \mapsto F(Y,A) \in \R^6$ denotes the function defining system \eqref{eq:system-S},
\item 
for any $t\in [0,T]$ and $A\in \mathcal{A}$, set $X(t,A)$ the value at time $t$ of the solution to the Cauchy problem $\diff{}{t} Y = F(Y,A)$ associated with the initial condition $Y(0)=X_0$,
\item 
with the previous lemma, this function is well defined, $X(t,A)\ge 0$ 
and 
with inequality (\ref{Gronwall}), there exists a constant $\mathcal{K}>0$ only depending on $T$ and $\mathcal{M}$ such that $0 \le X(t,A) \le \mathcal{K}$,
\item 
  we then consider the function 
  $X:~[0,T]\times\mathcal{A}\to \R^6$
  and we consider the usual notations:
  $X(t,\cdot):~\mathcal{A}\to \R^6$ is the function $A \in\mathcal{A}\mapsto X(t,A)$ and $X(\cdot,A):~[0,T]\to\R^6$ is the function $t\in [0,T]\mapsto X(t,A)$.
\end{itemize}
As a consequence, the fonction $F$ of system (\ref{eq:system-S}) 
 can be restricted to $(-\mathcal{K},\mathcal{K})^6 \times  \mathcal{A}$, or its closure $[-\mathcal{K},\mathcal{K}]^6\times \overline{\mathcal{A}}$ when a compact set is needed.
\begin{lemma}
  \label{lem:F-glob-lipschtz}
  There exists $\mathcal{L}>0$, only depending on $T$ and $\mathcal{M}$, such that the fonction $F$ is $\mathcal{L}$-Lipschitz continuous on $[-\mathcal{K},\mathcal{K}]^6\times \overline{\mathcal{A}}$.
\end{lemma}
\begin{proof}
  It has already been proven in lemma \ref{lem:ineq-F} that $F$ is locally Lipschitz in $X$, it is therefore globally Lipschitz in $X$ on the compact $[-\mathcal{K},\mathcal{K}]^6\times \overline{\mathcal{A}}$.
Let us prove that it is also globally Lipschitz in $A$: the variable $X$ is assumed to be in $[-\mathcal{K},\mathcal{K}]^6$ and
two parameters $A_1,A_2 \in \overline{\mathcal{A}}$ are considered.
\\
The function $\tau_N$ in (\ref{tau_N}) satisfies,
\begin{align*}
  \big|\tau_N(X,A_1)
  -\tau_N(X,A_2)\big|
  \leq  
  \mathcal{K}\mathcal{M}|V^{N,1}_{\max}-V^{N,2}_{\max}|        
  + \mathcal{M}^3\mathcal{K} |K_{N}^1-K_{N}^2|
\end{align*}
and the same inequality holds for $\tau_C$ in (\ref{tau_C}) by changing $N$ into $C$.
Thus, $F_1$ and $F_2$ are globally Lipschitz in $A$.
With lemma \ref{lem:FonctionA}, the function $\tau_D$ in (\ref{tau_D})
satisfies,
\begin{align*}
  \big|\tau_D(X,A_1)
  -\tau_D(X,A_2)\big|
  \leq  
  \vert \mu_D^1- \mu_D^2\vert 
  + 2\mathcal{M}^2|Q^{C,\,1}_{\rm min}-Q^{C,\,2}_{\rm min}| 
  + 2\mathcal{M}^2|Q^{N,\,1}_{\rm min}-Q^{N,\,2}_{\rm min}|,
\end{align*}
and a similar inequality holds for $\sigma_N$ in (\ref{sigma_N}), consequently $F_3$ and $F_5$ are globally Lipschitz in $A$.
Using again lemma \ref{lem:FonctionA}, the function $\tau_M$ in (\ref{tau_M})
satisfies,
\begin{align*}
  \big|\tau_M(X,A_1)
  -\tau_M(X,A_2)\big|
  \leq  
  \vert \mu_M^1- \mu_M^2\vert 
  + 2\mathcal{M}^2|Q^{C,\,1}_{\rm min}-Q^{C,\,2}_{\rm min}| 
\end{align*}
and a similar inequality holds for $\sigma_M$ in (\ref{sigma_M}), consequently $F_4$ and $F_6$ are globally Lipschitz in $A$, which ends the proof.
\end{proof}
\vspace{10pt} \indent
As a direct consequence, we have the following first regularity result.
\begin{lemma} 
  \label{lem:X-weak-differentiable}
  Set $\mathcal{Q}=(0,T)\times \mathcal{A}$,
  $X$ and $\partial_t X$ are bounded Lipschitz-continuous functions on $\overline{\mathcal{Q}}$, 
and therefore belong to $W^{1,\infty}(\mathcal{Q})^6$ \cite[Thm. 4.5]{EvansGar}.
  \\[0.1cm]
  In particular, the application $A \in \mathcal{A} \mapsto X(\cdot,A) \in [C^1([0,T])]^6$ is Lipschitz-continuous.
\end{lemma}
\begin{proof}
  Firstly $X$ and $\partial_t X$ are bounded on $[0,T]\times \mathcal{A}$,
  \begin{displaymath}
    \|X(t,A)\| \leq \mathcal{K},
    \quad \|\partial_t X(t,A)\| \leq \|F\|_\infty,
  \end{displaymath}
  then, 
  for any $A,B\in\mathcal{A}$ and $s,t\in[0,T]$,  
  \begin{displaymath}
    \|X(t,A)-X(s,A)\| \leq \int_s^t \|F(X(\tau,A),A)\|d\tau \leq \|F\|_\infty(t-s)
  \end{displaymath}
  and so $X$ is Lipschitz in $t$. Moreover, with lemma 
  \ref{lem:F-glob-lipschtz}
  \begin{displaymath}
    \|\partial_t X(t,A)-\partial_t X(s,A)\| =  \|F(X(t,A),A)-F(X(s,A),A)\| 
      \leq \mathcal{L}\|F\|_\infty(t-s),
  \end{displaymath}
  and $\partial_t X$ is also Lipschitz in $t$. We now prove that $X$ is also Lipschitz in $A$, 
  \begin{align*}
    \|X(t,A)-X(t,B)\| &= \|\int_0^t  F(X(s,A),A)-F(X(s,B),B) \ds\| 
    \\
    & \leq 
      \mathcal{L}\int_0^t\|X(s,A)-X(s,B)\| + \|A-B\|\ds,
  \end{align*}
  and therefore, Gronwall's lemma yields for ant $t \in [0,T]$,  
  \begin{displaymath}
    \|X(t,A)-X(t,B)\| \leq  \mathcal{L}Te^{\mathcal{L}T}\|A-B\|.
  \end{displaymath}
  This implies that $\partial_t X$ is Lipschitz in $t$,
  \begin{align*}
    \|\partial_tX(t,A)-\partial_tX(t,B)\| 
    &= \|F(X(t,A),A)-F(X(t,B),B)\| 
    \\
    &\leq 
    \mathcal{L}\Big[ \|X(t,A)-X(t,B)\| + \|A-B\|\Big],
  \end{align*}
  which ends the proof.
\end{proof}
\vspace{10pt} \indent
Higher regularity is obtained for $X:~A\mapsto X(\cdot,A)$ 
with values in $C(\overline{\mathcal{A}})$ 
which is differentiable in $[0,T]$ as now stated.
\begin{lemma}
  \label{lem:stablilite}
  The function $X$ satisfies, 
  $X \in C^1([0,T],C^0(\overline{\mathcal{A}})^6) \cap {\rm Lip}([0,T],W^{1,\infty}(\mathcal{A})^6)$.
\end{lemma}
\begin{proof}
By previous lemma, for any $p\in [1,+\infty)$, $X,\partial_t X \in L^\infty(0,T,W^{1,p}(\mathcal{A})^6)$.
\\
Note that since $L^\infty(\mathcal{A})$ is not separable, the measurability of $X$ (resp. $\partial_t X$) as a mapping from $(0,T)$ to $W^{1,\infty}(\mathcal{A})^6$ may fail. 
\\[0.1cm]
But $X \in W^{1,\infty}(0,T,W^{1,p}(\mathcal{A})^6)$ for any finite $p$, and, if $0\leq s<t \leq T$, one has (\cite{Brezis}) that
{\small \begin{align*}
          \|X(t,\cdot)-X(s,\cdot)\|_{W^{1,p}(\mathcal{A})} \leq& \int_s^t \|\partial_t X(\tau,\cdot)\|_{W^{1,p}(\mathcal{A})}d\tau  \leq (t-s)^{\frac{p-1}{p}}\|\partial_t X\|_{L^p(0,T,W^{1,p}(\mathcal{A}))} 
          \\ \leq&  (t-s)^{\frac{p-1}{p}}\|\partial_t X\|_{W^{1,p}(\mathcal{Q})} \leq (t-s)^{\frac{p-1}{p}}\|\partial_t X\|_{W^{1,\infty}(\mathcal{Q})} [19|\mathbb{A}|]^{\frac1p}
        \end{align*}
      }
      Since $X(t,\cdot),X(s,\cdot) \in W^{1,\infty}(\mathcal{A})^6$, 
      $\lim_{p \to +\infty}\|X(t,\cdot)-X(s,\cdot)\|_{W^{1,p}(\mathcal{A})}=\|X(t,\cdot)-X(s,\cdot)\|_{W^{1,\infty}(\mathcal{A})}$, and passing to the limit $p \to +\infty$ yields 
      \begin{align*}
        \|X(t,\cdot)-X(s,\cdot)\|_{W^{1,\infty}(\mathcal{A})} \leq \|\partial_t X\|_{W^{1,\infty}(\mathcal{Q})} (t-s).
      \end{align*}
Thus, $X$ is a bounded Lipschitz-continuous function from $[0,T]$ to $W^{1,\infty}(\mathcal{A})^6$. 
      \\[0.1cm]
      By selecting $p$ such that $W^{1,p}(\mathcal{A}) \hookrightarrow C^0(\overline{\mathcal{A}})$, one gets for any $t,h$ such that $t,t+h \in [0,T]$, 
      \begin{align*}
        &\| \frac{X(t+h,.)-X(t,.)}{h}-\partial_t X(t,.)\|_{C(\overline{\mathcal{A}})} = \frac1{|h|} \| \int_t^{t+h}\partial_tX(\tau,.)-\partial_t X(t,.) \, d\tau\|_{C(\overline{\mathcal{A}})}
        \\ \leq &
                  \frac1{|h|}  \int_{(t,t+h)}\|\partial_tX(\tau,.)-\partial_t X(t,.)\|_{C(\overline{\mathcal{A}})} \, d\tau 
                  \leq \mathcal{L}\|F\|_\infty h.
      \end{align*}

      So, $X$ as a function with values in $C(\overline{\mathcal{A}})$ is differentiable in $[0,T]$, with derivative $\partial_t X$ and the result in  lemma \ref{lem:stablilite} holds.
\end{proof}
\vspace{10pt}\indent
We are now interested in the differentiation of the function $X(t,A)$
with respect to $A$. 
Since $F$ is not a continuously differentiable function, one cannot apply the implicit function theorem and our arguments will be based on the chain rule in Sobolev spaces. 
In oder to proceed, $F$ first need to be extended to a Lipschitz function on $\R^{6+15}$. 
This can be obtained by Kirszbraun theorem, but for technical reasons, it will be done by replacing each coordinate $x$ of $A$ in $F$ by $g(x)=\max(\frac1{\mathcal{M}},\min(\mathcal{M},x))$ and each coordinate $x$ of $X$ by $h(x)=\min(\mathcal{K},x^+)$ (or a regularization of these functions).
Doing so, each $F_k$ ($k=1,...,6$) is a globally Lipschitz-continuous, piecewise $C^1$ function in the sense of Murat \textit{et al} \cite{Murat}:
"\textit{there exists a finite Borel-partition $(P^i)_{i\in I_k}$ of $\R^{6+15}$ and the same number of globally Lipschitz-continuous $C^1$ functions $\widetilde F^i_k$ on $\R^{6+15}$ such that $\widetilde F^i_k= F_k$ in $P^i$.}" 
\\[0.2cm]
Thus,  \cite[Thm. 2.1]{Murat}, since $X \in W^{1,p}(\mathcal{Q})^{6}$ for any finite $p$, one gets that $F_k(X,\cdot) \in W^{1,p}(\mathcal{Q})$ with the chain rule: a.e. in $\mathcal{Q}$, 
\begin{align*}
  \nabla_{(t,A)} F_k(X,\cdot)  =& \sum_{i \in I_k} 1_{P^i}(X) (D F_k)(X) \nabla_{(t,A)} X.
\end{align*}
Moreover, $ Y \in W^{1,p}(\mathcal{Q})^{6} \mapsto  F_k(Y,\cdot)\in W^{1,p}(\mathcal{Q})$ is sequentially continuous for the weak convergences and continuous for the strong topologies.
\\[0.1cm]
As the equality $\partial_tX(t,A) = F\big(X(t,A),A\big)$ holds in  $W^{1,p}(\mathcal{Q})$, one gets that 
\begin{align*}
  \partial_t\nabla_AX(t,A) 
  &= 
  \nabla_A\partial_tX(t,A) 
    = 
  \nabla_A F\big(X(t,A),A\big) 
  \\&= 
  \Big(\sum_{i \in I_k} 1_{P^i}\big(X(t,A)\big) (D F_k)\big(X(t,A),A\big) \nabla_{A} X(t,A)\Big)_k
\end{align*}
\textit{a priori} in $L^{p}(\mathcal{Q})^{6\times 15}$ since $\nabla$ is a continuous linear map from $W^{1,p}(\mathcal{Q})^{6\times 15}$ onto $L^{p}(\mathcal{Q})^{6\times 15}$, then in $W^{1,p}(\mathcal{Q})^{6\times 15}$ thanks to the above regularity given by the chain-rule.
      \\[0.1cm]
For any 
$
  Z \in(a,\Nin,\Cin,V^N_{\max}, K_N, V^C_{\max}, K_C, \QCMin, \QNMin, \alpha, \mu_D, \mu_M, \Theta_D, \Theta_M, m_D),
$
\begin{align*}
  \partial_Z F_k\big(X(t,A),A\big)  =& \mathbb{G}_{k,Z}\big(X(t,A),A\big) + 
\mathbb{H}_k\big(X(t,A),A\big)\partial_Z X(t,A)
\end{align*}
where 
$\mathbb{G}_k=(\mathbb{G}_{k,a},\cdots,\mathbb{G}_{k,\Theta_M})^T$ 
denotes $D_AF_{k}$ 
and $\mathbb{H}_k$ denotes 
$D_XF_{k}$ 
are  bounded Borel functions.
\\
Thus, back to the evolution problem, 
\begin{align}
  \notag
  \partial_t \nabla_A X (t,A)
  & =
    \Big(
    \mathbb{G}_k\big(X(t,A),A\big) + \mathbb{H}_k\big(X(t,A),A\big)\nabla_A X(t,A)
    \Big)_k
  \\ \label{ODENabla}
  &=
    D_AF\big(X(t,A),A\big) + D_XF\big(X(t,A),A\big) \nabla_A X(t,A)
\end{align}
which  ODE is solvable 
(\cite[Caratheodory Conditions for Time-Varying Vector Fields]{Cortes}) 
\\
Finally remark that, since we restricted our study to the case of a constant initial condition $X_0$, the initial condition for the above ODE is  $\nabla_A X(0,A)=0$
leading to the following lemma.
\begin{lemma}
  \label{lem:differentiabilite}
  $\nabla_A X(\cdot,A)$ is the unique absolute-continuous function in $[0,T]$, solution to the linear ODE \eqref{ODENabla} associated with the initial condition 
    $\left (\nabla_A X\right )_0 = 0$.
\end{lemma}
\section{Numerical simulations: an inverse problem}\label{Numerical_simulations}
We wish in this section to investigate the ability of 
determining system (\ref{eq:system-S}) parameters (supposed to be constant) with data obtained from chemostat experiments.
Up to now such data are not available: a target model will be considered instead and numerical experiments will be used to measure the capacity of recovering the target model parameters.
For this, virtual chemostat data are constructed by numerically solving system (\ref{eq:system-S}) and recording its variable states with a given sampling rate.
\\
Attention has been paid, as much as possible, to build a realistic and biologically relevant target problem where TEP production is encountered.
So far, it is not possible to find a set of parameter for  system (\ref{eq:system-S})  in the literature for a given species of diatoms: some parameters are known (usually within a given range) for a specific species, other parameters can be evaluated indirectly.
This is detailed in section \ref{sec:target-problem}, variable units and parameters units and values are reported in tables 
\ref{tab:medel-var} and \ref{tab:medel-param} respectively.
\\
The parameter identification algorithm is presented in section 
\ref{parameter-identification} together with its numerical evaluation.
\subsection{The target problem}
\label{sec:target-problem}
\small{
\begin{table}[h!]
  \centering
  \caption{Variable units for system (\ref{eq:system-S}) (Xeq. is for \textit{Xanthum gum equivalent}, a standard to express TEP concentration).  }
  \vspace{7pt}
  \begin{tabular}{|c|l|l|}
    \hline
    Variable & description & unit
    \\ \hline
    $C$ & concentration in carbon nutrient & $\mu\,$mol/L
    \\
    $N$ & concentration in nitrogen nutrient & $\mu\,$mol/L
    \\
    $Q_C$ & cellular quota in carbon & $10^{-9}\mu\,$mol/cell
    \\
    $Q_N$ & cellular quota in nitrogen & $10^{-9}\mu\,$mol/cell
    \\
    $D$ & diatom concentration  & $10^{9} $ cell/L
    \\
    $M$ & TEP concentration  & g Xeq./L  
    \\  \hline
  \end{tabular}
  \label{tab:medel-var}
\end{table}
}
In this section is presented a \textit{target problem} for parameter identification
with physical variable units and with parameters in agreement with biological measurements.
\\ 
Problem (\ref{eq:system-S}) is considered with 
the variable units given in table \ref{tab:medel-var}
and with the parameters given in table \ref{tab:medel-param}.
\subsubsection{Parameter definition}
Parameter estimation for model (\ref{eq:system-S}) is not obvious: several parameters have been found in the literature (details and precise references follow)
but parameter values differ from a species to another and sometimes depend on the conditions and on the authors. A global set of parameters for a given species of diatom is not available and we will rather focus on searching order of magnitude:
without explicit mention, parameters are set to intermediate values inside range of values. 
Moreover some specific parameters could not be found in the literature and self estimations then will be considered.
\small{
\begin{table}[h!]
  \centering
  \caption{Parameters for system (\ref{eq:system-S}) }
  \vspace{7pt}
   \begin{tabular}{|c|l|l|l|}
    \hline
    Variable & description & value  & unit
    \\ \hline
    $a$ &  chemostat dilution rate&  0.59 & day$^{-1}$
    \\
    $\Cin$ &  input $C$ concentration & 2000 &$\mu\,$mol/L
    \\
    $\Nin$ &  input $N$ concentration &15  &$\mu\,$mol/L
    \\
    $V^C_{\max}$ &  max.  $C$ uptake rate& 400 & $10^{-9}\,\mu\,$mol/cell/day
    \\
    $V^N_{\max}$ &  max.  $N$ uptake rate& 70 & $10^{-9}\,\mu\,$mol/cell/day
    \\
    $K_C$ &  $C$ uptake half saturation constant & 1.5 &$\mu\,$mol/L
    \\
    $K_N$ &  $N$ uptake half saturation constant & 1.25 &$\mu\,$mol/L
    \\
    $m_D$ & diatom mortality rate  &  0.1 & day$^{-1}$
    \\
    $\QCMin$ &   min $C$  cellular quota&  0.5 & $10^{-9}\,\mu\,$mol/cell
    \\
    $\QNMin$ &  min $N$  cellular quota& 10 &  $10^{-9}\,\mu\,$mol/cell
    \\
    $\mu_D$ & max.  diatom growth rate&  1.24 & day$^{-1}$
    \\
    $\mu_M$ & max. TEP production rate &  8.2 & $10^{-9}\,$ g Xeq./cell/day  
    \\
    $\Theta_D$ &  max. $N$ consumption rate for diatom growth& 4.5 &$10^{-9}\mu\,$ mol/cell/day
    \\
    $\Theta_M$ &  max. $C$ consumption rate for TEP production & 1000 &$10^{-9}\mu\,$ mol/cell/day
    \\
    $\alpha$ &  $C:N$ stoichiometric ratio  &  16 &  without unit
    \\ \hline
  \end{tabular}
  \label{tab:medel-param}
\end{table}
}
\paragraph{Nitrogen uptake.} 
Nitrogen is up-taken by diatoms among three inorganic sources: ammonium, nitrate and nitrite. Only nitrate uptake is  considered here, as it usually is the main source of nitrogen, see \textit{e.g.} \cite{serra1978nitrate}.
To increase the effects of nitrogen limitation during bloom, a low value of $\Nin$ = 15  $\mu\,$mol/L is set, relatively to the range of nitrate input concentration for chemostat experiments of 14-240 $\mu\,$mol/L given in \cite{Klausmeier}.
In \cite[table A46]{joergensen1979handbook} are reported  half saturation constants for nitrate uptake for various species of marine diatoms which range from 0.45 to 1.87 
 $\mu\,$mol/L.
 Maximal  nitrate uptake rate for various species of diatoms can be found in 
\cite[table A60]{joergensen1979handbook} ranging from 72 to 384  $10^{-9}\,\mu\,$mol/cell/day and a low value of $V^N_{\max}=70 ~10^{-9}\,\mu\,$mol/cell/day is set here to enhance the effects of nitrate limitation after a bloom.
Minimal cellular quota in nitrate 
$\QNMin= 45.4 \; 10^{-9}\,\mu\,$mol/cell
is used in \cite{Klausmeier}, 
however this value is  too high to allow TEP production
and a smaller value of $\QNMin= 10 \; 10^{-9}\,\mu\,$mol/cell 
will be arbitrary considered here.
\paragraph{Carbon uptake.} 
Carbon nutrient is available either as dissolved CO$_2$
or as HCO$_3^-$: dissolved CO$_2$ has a smaller and variable concentration, between 5 and 25 $\mu\,$mol/L whereas HCO$_3^-$ has a stable and much higher concentration of 2000 $\mu\,$mol/L, we will then only consider HCO$_3^-$ as source of carbon
with an input concentration of $\Cin$=2000 $\mu\,$mol/L.
In \cite{burkhardt2001co2},
the half saturation constant for HCO$_3^-$ uptake 
has been evaluated to 1-2  $\mu\,$mol/L and the 
maximal   HCO$_3^-$ uptake rate to 175 - 230 $\mu\,$mol/(mg$\,$Chl$_{\rm a}$)/h  
for diatom \textit{Thalassiossira Weissflogii} (Chl$_{\rm a}$ standing for type a chlorophyll). With \cite[table A30]{joergensen1979handbook} the cellular content of Chl$_{\rm a}$ is of 8.1 $10^{-11}$ mg for \textit{Thalassiossira} species and we get a range of 340 to  450  $10^{-9}\,\mu\,$mol/cell/day for $V^C_{\max}$.
The minimal cell quota in carbon $\QCMin$ could not have been found in the literature.
We propose an estimation based on \cite{johnston1996inorganic} 
where it is reported 
that diatoms are able to accumulate inorganic carbon only when its external concentration is low (when it is of 200 $\mu\,$mol/L the cell concentration in inorganic carbon is of 400 $\mu\,$mol/L),
but meanwhile unable to accumulate inorganic carbon when  its external concentration is at its normal state of 2000 $\mu\,$mol/L (then the cell concentration in inorganic carbon remains above 1470 $\mu\,$mol/L).
Assuming a cell volume of 750 $\mu\,$m$^3$, see \textit{e.g.} \cite{chen2011effect}, this suggests that $\QCMin \ge 0.3\,10^{-9}\,\mu\,$mol/cell and that $\QCMin \le 1.5\,10^{-9}\,\mu\,$mol/cell.
\paragraph{Other parameters.} 
The chemostat dilution rate is set to $a=$0.59 day$^{-1}$ as for the chemostat numerical simulations in \cite{Klausmeier}, the C:N Redfield ratio is $\alpha=16$ and the mortality rate is set to 0.1 day$^{-1}$, a generic value for phytoplankton species 
given in \cite[table A72]{joergensen1979handbook}.
The maximal growth rate for diatom \textit{Thalassiossira Weissflogii}
is set to $\mu_D=1.24$ day$^{-1}$ as measured in \cite{garcia2012effect},
and chemostat experiments of TEP production in \cite{fukao2012effect}
show a maximal TEP production rate of  
$\mu_M = 8.2 \;10^{-9}\,$ g Xeq./cell/day  
for a marine diatom in the pacific ocean.
\\
The two remaining parameters $\Theta_D$ and $\Theta_M$ could not be found in the literature.
The maximal nitrate consumption rate for diatom growth $\Theta_D$ can
be estimated from the diatom maximal growth rate  $\mu_D=1.24$ day$^{-1}$.
In \cite[table A13]{joergensen1979handbook} is reported the carbon biomass for diatom \textit{Thalassiossira Weissflogii} which ranges between 
$250-450 \;10^{-9}\,\mu\,$mol C/cell, 
thus the 
maximal C consumption rate for diatom growth is of order $250-350\times 0.24 \;10^{-9}\,\mu\,$mol C/cell/day 
and using C:N  Redfield ratio for diatom growth 
$\alpha=16$ we propose the estimation 
$\Theta_D = 250-450\times 0.24/16 \;10^{-9}\,\mu\,$mol N/cell/day = 3.75-6.75 $\;10^{-9}\,\mu\,$mol N/cell/day.
\\
Finally no satisfactory estimation of $\Theta_M$ could be made and it 
will be set to $\Theta_M = 1000 \;10^{-9}\mu\,$ mol/cell/day 
leading to a maximal TEP concentration of order 0.1 g Xeq./L which is consistent with the experiments in \cite{fukao2012effect}.
\subsubsection{Numerical simulation}
\begin{figure}[h!] 
  \centering
  \input{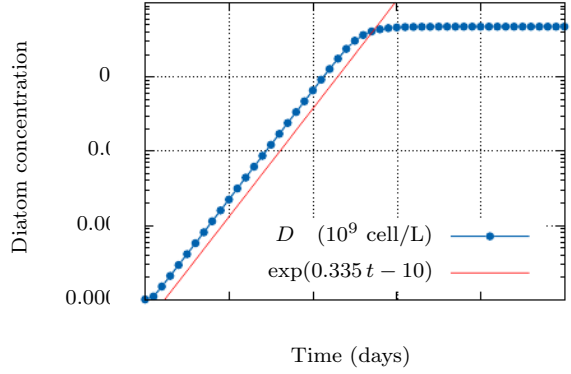}
  \caption{ Evolution of  diatom population in log scale}
  \label{fig:simu-sytem-S-D_log}
\end{figure}
\begin{figure}[h!]
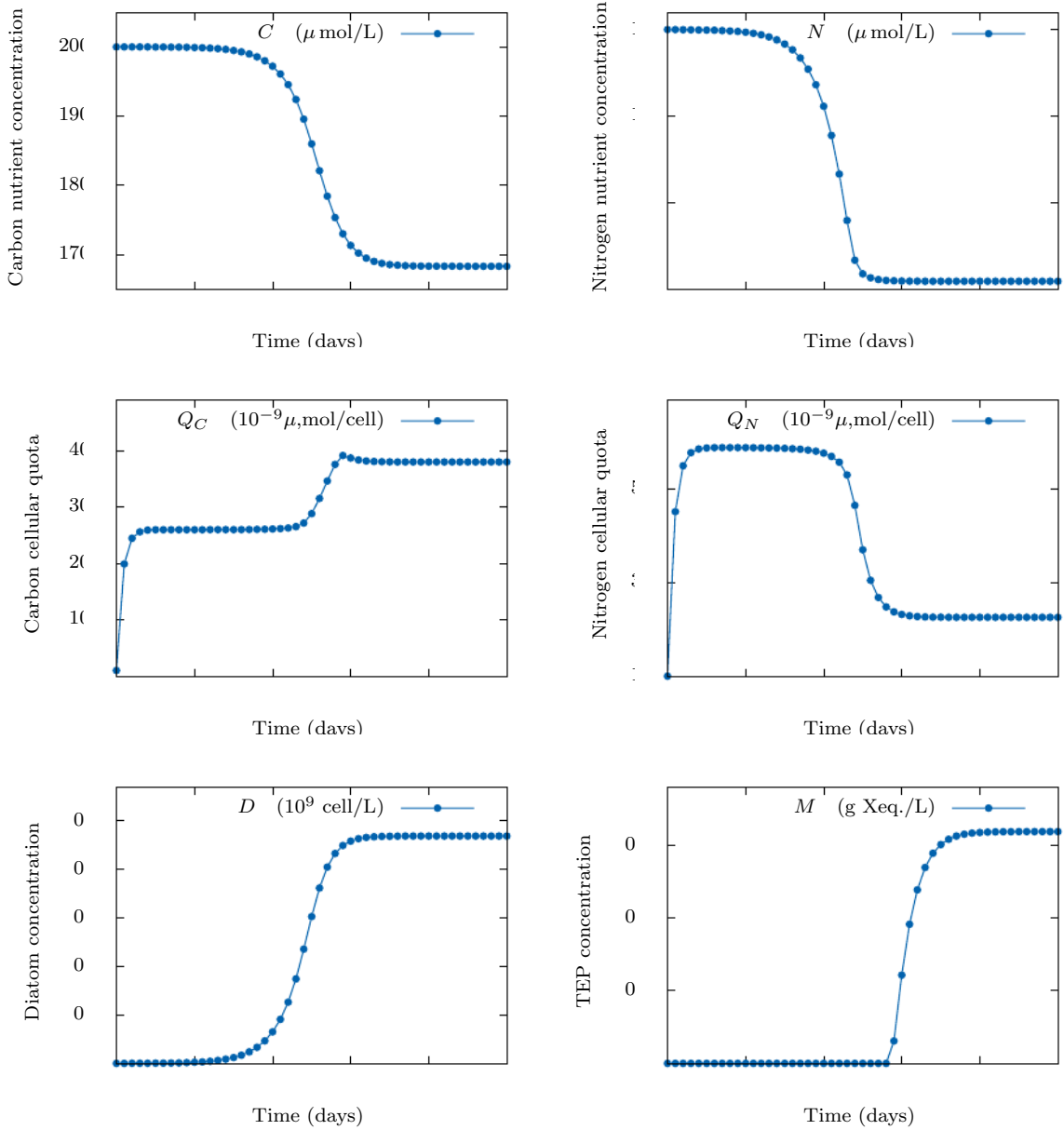
 
  \centering
  \mbox{
    \input{GV_2020_C.txt}
    ~~~~~~~~\input{GV_2020_N.txt}
  }
  \\
  \mbox{
    \input{GV_2020_QC.txt}
    ~~~~~~~~\input{GV_2020_QN.txt}
  }
  \\
  \mbox{
    \input{GV_2020_D.txt}
    ~~~~~~~~\input{GV_2020_M.txt}
  }
  \caption{ Evolution of system (\ref{eq:system-S}) with the parameters in table \ref{tab:medel-param}.}
  \label{fig:simu-sytem-S}
\end{figure}
The evolution of system (\ref{eq:system-S}) for the parameters in table \ref{tab:medel-param} is depicted on figures \ref{fig:simu-sytem-S} and \ref{fig:simu-sytem-S-D_log}.
System (\ref{eq:system-S}) is solved on a period of time of 50 days and the solution is recorded every days. At initial state the chemostat concentrations in nutrients $C(0)$ and $N(0)$ are set to $\Cin$ and $\Nin$, diatom cellular quota $Q_C(0)$ and $Q_N(0)$ are set to their associated minimal quota $\QCMin$ and $\QNMin$, the initial diatom concentration is $D(0)=10^5$ cell/L (roughly 1/500 of diatom population observed after the bloom) and no TEP is initially present in the chemostat ($M(0)=0$).
\\ \\ \indent
Figure \ref{fig:simu-sytem-S-D_log} specifically shows the exponential growth of diatom population up to a final plateau, with a growth rate numerically evaluated to 0.335.
\\ \\ \indent
Figure \ref{fig:simu-sytem-S} shows the global evolution 
of system (\ref{eq:system-S}) for its six variables $C$, $N$, $Q_C$, $Q_N$, $D$ and $M$.\\
In the 10 first days, nutrient supplies are 
sufficiently abundant to feed diatoms (whose  population is multiplied by 20) and  nutrient concentrations remain almost equal to the input concentration, meanwhile cellular quotas 
quickly increase to reach a first plateau. From day 10 to day 30 diatom population is again multiplied by 20, nutrient supplies are still abundant enough to feed the diatom though a slight decrease in  nutrient concentrations is observed, cellular quotas stay at their first plateau level.
Between day 20 and day 30 is observed the end of the exponential growth of  diatom population which now increases by a factor 7 only: nitrate becomes limiting, its concentration falls to roughly 1/30 of the input concentration stopping diatom growth, nitrate cellular quota decreases to reach a second plateau; meanwhile carbon nutrient remains abundant, carbon uptake is no longer balanced by cellular growth and cellular quota in carbon increases again. In the following period, between days 30 and 40, carbon cellular quota reaches a peak and stop increasing: TEP production is initiated when $Q_C/Q_N>\alpha$, carbon uptake is balanced by  TEP production and $Q_C$ reaches its second plateau level. In the last period of time, TEP reaches its maximal concentration and  the system stays in an equilibrium state.
\subsection{Parameter identification}
\label{parameter-identification}
System (\ref{eq:system-S}) is rewritten as,
\begin{equation}
  \label{eq:ode-param}
  \dfrac{\d{X}}{\dt} = F(P,X), \quad X(0) = X^0,
\end{equation}
with variable $X=(X_i)_{1\le i\le 6} = (C,N,Q_C, Q_N, D, M)$, with the  initial condition $X^0$ introduced in section \ref{sec:target-problem} ($X^0=(\Cin, \Nin, \QCMin, \QNMin, D(0)=10^5\,{\rm cell/L}, M(0)=0)$) and
where $P\in \R^p$  are the model parameters assumed to be constant: the vector $P$ is composed of the 15 model parameters excepted $a$, $\Cin$, $\Nin$ and $\alpha$ which are considered as problem data, so that $P\in \R^p$ with $p=11$.
\\
The target solution $X^{\rm tg}$ is the solution of problem (\ref{eq:ode-param}) with the parameters $P^{\rm tg}$ in table \ref{tab:medel-param}.
The same final time $T=50$ j as on figure \ref{fig:simu-sytem-S} is considered. A sampling period $\delta t = T /m$ (for an integer $m$) is introduced and the associated sampling time instants are the $t_j = j \delta t$ for $j=1, \cdots, m$.
Our goal will be to minimise the following cost functional $\mathcal{L}$,
\begin{equation}
  \label{eq:func-L}
  \mathcal{L}(P) = \sum_{j=1}^{m} \sum_{i=1}^{6} r_{ij}(P)^2,
  \quad  r_{ij}(P) := X_i(t_j) - X^{\rm tg}_i(t_j)
\end{equation}
which obviously has value 0 when $P=P^{\rm tg}$.
\\
The key argument is that, following theorem 
\ref{theo:main-result}, it is possible to differentiate $X(t)$ with respect to $P$ and thus to compute a gradient for the functional $\mathcal{L}$.
More precisely, 
by differentiating the ODE in \ref{eq:ode-param},
$\dfrac{\partial  X}{\partial P}$
 satisfy the  following ODE,
\begin{equation}
  \label{eq:ode-dX_dP}
  \dfrac{\d}{\dt}\left (  \dfrac{\partial  X}{\partial P}\right )
  =
  \dfrac{\partial F}{\partial P} (X,P) + \dfrac{\partial F}{\partial t} (X,P)\dfrac{\d{X}}{\d{X}},
  \quad \dfrac{\partial  X}{\partial P}(0) = 0,
 \end{equation}
where it has to be noticed that the initial condition $\dfrac{\partial  X}{\partial P}(0) = 0$ follows from the fact that the same initial condition $X(0) = X^0$ is imposed for $X(t)$ independently with $P$.
\\
Then, the derivative $\dfrac{\partial X}{\partial P} (t)$ will be computed numerically by solving the coupled ODE system \ref{eq:ode-param}~\ref{eq:ode-dX_dP}.
Thus, the evaluation of $\nabla \mathcal{L}(P)$ requires the resolution of the ODE system \ref{eq:ode-param}~\ref{eq:ode-dX_dP} over the time interval $[0,T]$, the storage of $X(t_j)$ and of $\partial X/\partial P(t_j)$ at every sampling time instants $t_j$ to get,
\begin{displaymath}
  \nabla \mathcal{L}(P) = 2 \sum_{j=1}^{m} \sum_{i=1}^{6} \left (
  X_i(t_j) - X^{\rm tg}_i(t_j) 
  \right )
  \dfrac{\partial  X_i}{\partial P}(t_j).
\end{displaymath}
The computation of $\nabla \mathcal{L}(P)$
will be held using the Runge Kutta 4 algorithm and 
a time step of $0.01$ day ensuring a very accurate 
computation of $X(t)$ and $\partial  X/\partial  P(t)$, 
for which the error has been evaluated to be below 1E-7 for the $L^\infty$ norm based on a numerical convergence analysis.
\subsection{Minimisation algorithm}
\paragraph*{The Gauss Newton method}
The  Gauss Newton method to solve a least square problem is briefly recalled.
Consider the functional to be minimized, for $P\in \R^p$,
\begin{displaymath}
  \mathcal{L}(P) = \mathbf{r}^T \mathbf{r}, \quad 
  \mathbf{r} = (r_i(P))_{1\le i \le n}.
\end{displaymath}
It is minimized starting from an initial guess $P^0$ 
and computing successive approximations $P^{i+1} = P^i + \delta$, 
where the increment $\delta$ is computed at each step as the solution of the linear system 
$J^TJ \delta = -J^T \mathbf{r}(P^i)$, 
where $J \in \R^{n\times p}$ is the Jacobian matrix of $\mathbf{r}$ at point $P_i$.
Note that the system matrix $J^TJ\in \R^{p\times p}$ is symmetric positive and moreover definite in practice.
\\
The algorithm is stopped when $\|\mathbf{r}\|_2 \le \varepsilon$, or more precisely in our case when 
\begin{equation}
\label{param-id-tol}
\dfrac{
  \mathcal{L}(P)^{1/2} 
}
{
| X^{\rm tg} |_2
}
\le {\rm tol}, \quad 
| X^{\rm tg} |_2^2 := \sum_{j=1}^{m} \sum_{i=1}^{6} X^{\rm tg}_i(t_j)^2  
\end{equation}
for a given tolerance ${\rm tol}$.
\\ \\
The advantage of Gauss Newton algorithm is that it is able to deliver very accurate numerical solutions together with very fast convergence properties. 
However its main drawback is to induce instabilities, especially when the initial guess is not accurate enough, leading to numerical blow-ups in the resolution of ODE (\ref{eq:ode-param}). A modified Gauss-Newton algorithm can be considered, as detailed below, to improve the method stability.
\paragraph*{Modified  Gauss Newton method}
At every step of the Gauss Newton method, the increment $\delta$
 is now considered as a descent direction.
ODE (\ref{eq:ode-param}) is numerically solved with parameters $P^i + h_k \delta$ with a decreasing $h_k$ starting at $h_0 = 1$ in order to detect the presence of instabilities and more precisely of a numerical blow-up.
Once a safe value of $h_k \in (0,1]$ has been determined, a golden section search is done to evaluate a best descent step $h \in (0, h_k]$ and the present algorithm step ends by setting $P^{i+1} = P^1 + h \delta$.
\\ \\
Instabilities remain problematic even with the modified version of the Gauss Newton method. Alternatively a classical gradient algorithm can also be used.
Such an algorithm is stable but after a few steps (2 or 3 steps in general) it stops moving, delivering numerical solutions of very poor accuracy.
It can however be considered in order to improve the initial guess quality.
\paragraph*{The minimisation method used in practice}
None of the previous methods (Gauss Newton, modified Gauss Newton and gradient method) being fully satisfactory, a combination of these methods will be considered.
The first remark is that the Gauss Newton method is very accurate and fast once the initial guess $P^0$ is close enough to the target solution $P^{\rm tg}$.
Then, the following combination of methods is used to build an accurate initial guess, and the algorithm ends with a Gauss Newton with a small tolerance set to ${\rm tol}=10^{-10}$ in \eqref{param-id-tol} and a maximal number of iterations set to 50.
\\
First of all we avoid performing directly the parameter identification on the whole time period $[0,T]$ and instead try to predict $P$ on a shorter time period $[0, T_{1/2}]$ (with $T_{1/2} = 20 $ days). Then these predicted parameters are used as an initial guess for a second parameter identification now on the time period $[0, T_{1/2}+1]$ and so on up to $T_{1/2}+k=T$.
\\
At every stage we start with a few steps (namely 2) of the gradient method.
Then the modified Gauss Newton method is applied, with at most 50 iterations and with a rough tolerance of $10^{-4}$ and eventually the Gauss Newton method is launched with also a maximal number of iterations of 50 and a tolerance of $10^{-4}$.
\subsection{Numerical results}
The algorithm presented in the previous section is evaluated by setting the initial guess components to
$P^0_i = P^{\rm tg}_i ( 1 + \theta_i) $ where the $\theta_i$ are randomly generated numbers of given amplitude $\varepsilon$, 
$\theta_i\in [-\varepsilon, \varepsilon]$.
Parameter identification is launched with this initial guess and we record:
\begin{enumerate}
\item whether it succeeded or failed to find a solution (failure being associated with a blow-up in the resolution of ODE (\ref{eq:ode-param}),
\item the final residual defined in \eqref{param-id-tol},
\item the maximal relative error $\|P^{\rm num} - P^{\rm tg}\|_\infty / \|P^{\rm tg}\|_\infty$ between the numerically determined parameters $P^{\rm num}$  and the target parameters.
\end{enumerate}
For a given amplitude $\varepsilon$, 
a number of 200 such experiments are done and the results are statistically presented in tables \ref{tab:res-1}, \ref{tab:res-0.5} and \ref{tab:res-0.25}
for three different sampling period of  1 day, half a day and a quarter of a day respectively.
{\small
\begin{table}[h!]
  \centering
  \begin{tabular}[h!]{l|c|c|c|c|c|c|c}
    Amplitude $\varepsilon$ 
    &1 \% & 5 \%& 10 \%& 15 \%& 20 \%&5 \% &50 \%
    \\ \hline    
    Success rate (in \%) 
    &87&88&72.5&55.5&42&32.5&19.5
    \\
    Mean residual 
    &1E-7&2E-7&9E-7&3E-7&0.4&0.2&0.5
    \\
    Mean max. error (in \%) 
    &5E-5&3E-4&2E-2&1.0&12.&29.&26.
  \end{tabular}
  \caption{Numerical results for a sampling period of 1 day}
  \label{tab:res-1}
\end{table}
\begin{table}[h!]
  \centering
  \begin{tabular}[h!]{l|c|c|c|c|c|c|c}
    Amplitude $\varepsilon$ 
    &1 \% & 5 \%& 10 \%& 15 \%& 20 \%&5 \% &50 \%
    \\ \hline    
    Success rate (in \%) 
    &76.5&85&67.5&57.5&49.5&37.5&11.5
    \\
    Mean residual 
    &6E-5&5E-6&6E-6&5E-6&4E-5&0.4&0.2
    \\
    Mean max. error (in \%) 
    &1E-3&1E-4&1E-3&2E-3&7E-2&13.&19.
  \end{tabular}
  \caption{Numerical results for a sampling period of 0.5 day}
  \label{tab:res-0.5}
\end{table}
\begin{table}[h!]
  \centering
  \begin{tabular}[h!]{l|c|c|c|c|c|c|c}
    Amplitude $\varepsilon$ 
    &1 \% & 5 \%& 10 \%& 15 \%& 20 \%&5 \% &50 \%
    \\ \hline    
    Success rate (in \%) 
    &84&90&76.5&65&49&37&14.5
    \\
    Mean residual 
    &1E-6&1E-6&2E-6&2E-6&1E-6&0.6&0.3
    \\
    Mean max. error (in \%) 
    &3E-4&3E-4&5E-4&3E-2&1E-2&14.&26.
  \end{tabular}
  \caption{Numerical results for a sampling period of 0.25 day}
  \label{tab:res-0.25}
\end{table}
}
\\ \\
A first fact is that, even for a small perturbation amplitude of 1 \%, there is no full guaranty of converging to a numerical solution because of instability's.
The algorithm success rate is quite large for small perturbation amplitudes but decreases and become quite small for a {larger amplitudes of 25 \%. 
The quality of the numerical solution can be easily controlled by the final residual: once the numerical solution is associated to a very small residual (below 1E-5, which usually occurs when the algorithm converges) then the maximal relative error between the numerical solution and the target parameters is very small, most generally below 1E-2 \%.
We point out that this figure is a max error and that in practice most of the parameters are determined with an accuracy below 1E-10 \% whereas few of them are less accurately predicted.
When the residual is higher, of order 1 or more, the parameters are less accurately predicted but still with a correct order of magnitude. Again, we point out that only a few of the parameters are not accurately predicted when most of them, with $\varepsilon = 50 \%$ \textit{e.g.}, are captured with a relative error below 1 \%.
This seems to indicate the presence of many local minima for the functional $\mathcal{L}$ around $P^{\rm tg}$.
\\
As expected, results are better when decreasing the sampling period, however, the benefits are not very important and a sampling period of 1 day seems to be relevant for this present application.
\\ \\
As a conclusion, parameter identification of model (\ref{eq:system-S}) seems to be feasible with chemostat experiment data.
The presented algorithm allows very accurate parameter identification but requires a first initial guess of high quality because of remaining strong unsuitability features.
It seems necessary to go to a coupled approach, first using a generic minimisation method for the functional $\mathcal{L}$ in (\ref{eq:func-L}) (such as a genetic algorithm \textit{e.g.}) able to predict a good initial guess, and then to end the parameter identification with a Gauss Newton type method. This is let to further developments on the subject.
\section{Acknowledgments}

The authors would like to thank Pr. Xavier Mari from IRD (Institute of Research for Development)  for valuable discussions and references he provided.
\\[0.2cm]
The authors would like to thank the National Label of Excellence I-SITE: Energy and Environment Solutions (E2S - UPPA) for the financial support given to this work. 


\end{document}